\pdfoutput=1
\RequirePackage{ifpdf}
\ifpdf % We~are running pdfTeX in pdf mode
\documentclass[pdftex]{sigma}%,draft
\else
\documentclass{sigma}
\fi

\usepackage{cleveref}
\usepackage{tensor}
\usepackage{tikz}

\newcommand{\bsqcup}{\text{\b{$\sqcup$}}}

\DeclareMathAlphabet{\eusm}{U}{}{}{} % Euler script math
\SetMathAlphabet\eusm{normal}{U}{eus}{m}{n}
\SetMathAlphabet\eusm{bold}{U}{eus}{b}{n}

\DeclareMathAlphabet{\eufrak}{U}{}{}{} % Euler fraktur math
\SetMathAlphabet\eufrak{normal}{U}{euf}{m}{n}
\SetMathAlphabet\eufrak{bold}{U}{euf}{b}{n}

\numberwithin{equation}{section}

\newtheorem{Theorem}{Theorem}[section]
\newtheorem{Lemma}[Theorem]{Lemma}
\newtheorem{Corollary}[Theorem]{Corollary}
\newtheorem{Proposition}[Theorem]{Proposition}
\newtheorem*{thm*}{Theorem}

\theoremstyle{definition}
\newtheorem{Definition}[Theorem]{Definition}
\newtheorem{Example}[Theorem]{Example}
\newtheorem{Remark}[Theorem]{Remark}
\newtheorem{notation}[Theorem]{Notation}

\begin{document}
\allowdisplaybreaks

\newcommand{\arXivNumber}{2203.05852}

\renewcommand{\thefootnote}{}

\renewcommand{\PaperNumber}{067}

\FirstPageHeading

\ShortArticleName{De Finetti Theorems for the Unitary Dual Group}

\ArticleName{De Finetti Theorems for the Unitary Dual Group\footnote{This paper is a~contribution to the Special Issue on Non-Commutative Algebra, Probability and Analysis in Action. The~full collection is available at \href{https://www.emis.de/journals/SIGMA/non-commutative-probability.html}{https://www.emis.de/journals/SIGMA/non-commutative-probability.html}}}

\Author{Isabelle BARAQUIN~$^{\rm a}$, Guillaume C\'EBRON~$^{\rm b}$, Uwe FRANZ~$^{\rm a}$, \newline
Laura MAASSEN~$^{\rm c}$ and Moritz WEBER~$^{\rm d}$}

\AuthorNameForHeading{I.~Baraquin, G.~C\'ebron, U.~Franz, L.~Massen and M.~Weber}

\Address{$^{\rm a)}$~Laboratoire de math\'ematiques de Besan\c{c}on, UMR 6623, CNRS,\\
\hphantom{$^{\rm a)}$}~Universit\'e Bourgogne Franche-Comt\'e, 16 route de Gray, F-25000 Besan\c{c}on, France}
\EmailD{\href{mailto:isabelle.baraquin@univ-fcomte.fr}{isabelle.baraquin@univ-fcomte.fr}, \href{mailto:uwe.franz@univ-fcomte.fr}{uwe.franz@univ-fcomte.fr}}

\Address{$^{\rm b)}$~Institut de Math\'ematiques de Toulouse, UMR5219, Universit\'e de Toulouse, CNRS,\\
\hphantom{$^{\rm b)}$}~UPS, F-31062 Toulouse, France}
\EmailD{\href{mailto:guillaume.cebron@math.univ-toulouse.fr}{guillaume.cebron@math.univ-toulouse.fr}}

\Address{$^{\rm c)}$~Formerly: RWTH Aachen University, Pontdriesch 10--16, 52062 Aachen, Germany}
\EmailD{\href{mailto:laura.maassen@t-online.de}{laura.maassen@t-online.de}}

\Address{$^{\rm d)}$~Saarland University, Fachbereich Mathematik, Postfach 151150,\\
\hphantom{$^{\rm d)}$}~D-66041 Saarbr\"ucken, Germany}
\EmailD{\href{mailto:weber@math.uni-sb.de}{weber@math.uni-sb.de}}

\ArticleDates{Received March 25, 2022, in final form August 31, 2022; Published online September 13, 2022}

\Abstract{We prove several de Finetti theorems for the unitary dual group, also called the Brown algebra. Firstly, we provide a finite de Finetti theorem characterizing $R$-diagonal elements with an identical distribution. This is surprising, since it applies to finite sequences in contrast to the de Finetti theorems for classical and quantum groups; also, it does not involve any known independence notion. Secondly, considering infinite sequences in $W^*$-probability spaces, our characterization boils down to operator-valued free centered circular elements, as in the case of the unitary quantum group $U_n^+$. Thirdly, the above de Finetti theorems build on dual group actions, the natural action when viewing the Brown algebra as a dual group. However, we may also equip the Brown algebra with a bialgebra action, which is closer to the quantum group setting in a way. But then, we obtain a no-go de Finetti theorem: invariance under the bialgebra action of the Brown algebra yields zero sequences, in $W^*$-probability spaces. On the other hand, if we drop the assumption of faithful states in $W^*$-probability spaces, we obtain a non-trivial half a de Finetti theorem similar to the case of the dual group action.}


\Classification{46L54; 46L65, 60G09}

\renewcommand{\thefootnote}{\arabic{footnote}}
\setcounter{footnote}{0}

\section{Introduction}\label{sec-intro}

In this work, we provide de Finetti theorems for the unitary dual group, also called the Brown algebra or the Brown--Glockner--von Waldenfels algebra.
De Finetti theorems have a long tradition in probability theory. In a nutshell, the aim is to characterize some notion of independence and a distribution law by distributional symmetries of a sequence of random variables. The question is how a symmetry object on the one side corresponds to a distributional statement on the other side.

More precisely, the classical de Finetti theorem states the following: A sequence $(x_i)_{i\in\mathbb N}$ of real-valued random variables is (conditionally) independent and identically distributed if and only if it is exchangeable, i.e., if and only if for all $n\in\mathbb N$ the distribution of $(x_1,\dots,x_n)$ is invariant under permutation. Hence, i.i.d.\ sequences are characterized by the action of the symmetric group $S_n$ and we may say that this is the distributional symmetry of classical independence.

Now, strengthenings on the side of symmetries -- for instance by passing to groups containing~$S_n$~-- implies certain distribution laws on the side of the sequence. Moreover, we may ask for de Finetti theorems for other types of independences. Such de Finetti theorems have been studied in various contexts going beyond the usage of groups as symmetry objects or classical independence from probability theory.

\subsection{Overview on some de Finetti theorems in free probability theory}

Let us briefly sketch some de Finetti theorems in free probability theory. Let $(x_i)_{i\in\mathbb N}$ be a~sequence of random variables in a noncommutative $W^*$-probability space satisfying certain assumptions specified below. We have
\begin{center}\renewcommand{\arraystretch}{1.2}%\setlength{\tabcolsep}{2.5pt}
\begin{tabular}{l|l|l|l}
Assumptions on $(x_i)$ &Distributional properties&Symmetry object&Ref.\\
\hline \hline
$x_i=x_i^*$, $x_ix_j=x_jx_i$&class. indep.&symm. group $S_n$&\cite{Definetti}\\
&&or spreadability&\\
\hline
$x_i=x_i^*$, $x_ix_j=x_jx_i$&class. indep., $\mathbb R$-Gaussian&orth. group $O_n$&\cite{Freedman}\\
\hline
$x_ix_j=x_jx_i$&class. indep., $\mathbb C$-Gaussian&unitary group $U_n$&\cite{Freedman}\\
\hline
\hline
$x_i=x_i^*$&free indep.& symm. qu. group $S_n^+$&\cite{koestler+speicher}\\
&&or qu. spreadability &\cite{curran2}\\
\hline
$x_i=x_i^*$&free indep., semicircular& orth. qu. group $O_n^+$&\cite{banica+curran+speicher}\\
\hline
no assumption&free indep., circular& unitary qu. group $U_n^+$&\cite{curran}\\
\hline
\hline
$x_i=x_i^*$&Boolean indep.& symm. qu. semigr.&\cite{liu}\\
&& or Bool. spreadability&\cite{liu2}\\
\hline
$x_i=x_i^*$&Boolean indep., Bernoulli& orth. qu. semigr.&\cite{liu3}\\
\hline
$x_i=x_i^*$&monotone indep.& mon. spreadability&\cite{liu2}\\
\hline
\hline
$x_i=\big(x_i^l, x_i^r\big)$&bi-free indep.&strongly qu. bi-invar.&\cite{freslon+weber}
\end{tabular}
\end{center}

Let us comment on the above table. Firstly, the classical results are well-known, an exposition may be found in \cite{kallenberg}. Note that exchangeability, as a characterization of i.i.d.\ sequences, may be relaxed to spreadability, i.e., the distribution of $(x_{i_1},\dots,x_{i_n})$ needs to be the same as that of $(x_1,\dots,x_n)$ for all $i_1<\dots<i_n$ and all $n$.

In the free case, K\"ostler and Speicher \cite{koestler+speicher} proved that free independence is characterized by quantum exchangeability using the quantum analog of the symmetric group, namely Wang's quantum permutation group $S_n^+$. This quantum symmetry has been relaxed to quantum spreadability by Curran~\cite{curran2} building on So{\l}tan's quantum families of maps~\cite{soltan}. On the other hand, the free de Finetti theorem has been strengthened in \cite{banica+curran+speicher} to several other quantum groups containing~$S_n^+$, for instance to $O_n^+$, but also to~$H_n^+$ and~$B_n^+$. Note that in free probability the semicircular distribution plays exactly the role of the real Gaussian in classical probability theory, for instance in terms of central limit theorems~\cite{nica+speicher}. The non-selfadjoint situation has been treated by Curran~\cite{curran}. In \cite{banica+curran+speicher}, there are also half-liberated versions of de Finetti theorems. Liu showed that intermediate quantum groups do not necessarily give strengthenings of the distributional properties~\cite{liu3}.

For Boolean independence, one has to employ quantum semigroup versions of the above quantum groups, see \cite{hayase, liu}. The crucial feature is that all algebras are non-unital. For instance, the relation $\sum_ku_{ik}=1$ in $C(S_n^+)$ is replaced by $\sum_ku_{ik}P=P$ for some projection $P$. Again, the Bernoulli distribution is the correct analog of the real Gaussian. There is no unitary version of a Boolean de Finetti theorem.

As for monotone independence, Liu \cite{liu2} proved a de Finetti theorem which adapts Curran's quantum spreadability to the monotone situation. There is no kind of quantum exchangeability for monotone independence and hence no quantum group like object involved. Also, there are no strengthenings to distributional descriptions of ``Gaussians''.

For bi-free independence, Freslon and the fifth author gave a de Finetti theorem~\cite{freslon+weber}. Here, we do not consider single variables $x_i$ but rather pairs $\big(x_i^l,x_i^r\big)$. The symmetry is basically given by the quantum permutation group $S_n^+$, but the action is more complicated. Moreover, the de Finetti theorem requires some technical assumption (the splitting property) which hopefully may be removed someday. Again, there is no ``Gaussian'' version of this de Finetti theorem.

\subsection{The role of combinatorics in the proofs: partitions of sets}

For the proofs of the above de Finetti theorems, a major role is played by the combinatorics underlying the respective independence concepts and the other distributional properties. More precisely, the proofs mainly go by decomposing the functional $\varphi$ of the noncommutative probability space into cumulants indexed by partitions of sets. Then, both the independence as well as the distribution (such as ``Gaussianity'') are reflected by the choice of the partitions. On the other hand, the algebraic relations of the quantum algebraic objects are also described by partitions. This provides the link and is the essence in the proofs of all de Finetti theorems.\looseness=1

We recall that classical independence is governed by all partitions of sets, the real Gaussian arises from a restriction to pair partitions, and in the complex case we have to involve two colors for the points of the partitions. The groups $S_n$, $O_n$ and $U_n$ obey exactly the same combinatorics. As for free independence and the corresponding quantum groups, all we have to do is to restrict to noncrossing (also called planar) partitions, in the Boolean case we use interval partitions, in the monotone case there are linearly ordered partitions, and bi-noncrossing partitions in the bi-free case.

\subsection{Different kinds of actions}

Secondly, an important feature of a de Finetti theorem is to specify the kind of action. While we have multiplicative actions in the classical and the free case, we must restrict to linear actions in the case of Boolean and bi-free independence. Note that for both free independence and bi-free independence, the symmetry object is the quantum permutation group $S_n^+$. However, in the first case, the action is a multiplicative one whereas in the second case, it is only linear (and also twisted). Thus, the right choice of the kind of action is an important ingredient in the precise formulation of de Finetti theorems.

\subsection{Further reading}

Besides the above mentioned articles related to de Finetti theorems, let us also mention the work by K\"ostler on various exchangeabilities \cite{gohm+koestler,gohm+koestler2,koestler}, and various de Finetti theorems in quantum information theory \cite{arnon-friedman+renner+vidick,brandao+harrow,lancien+winter} or quantum mechanics \cite{christandl+konig+mitchison+renner,hudson}.
See also \cite{Maassen2021,MR3448336} for further de Finetti theorems in the context of compact quantum groups.

\section{Main results}

In the present article, we give a number of de Finetti theorems for the unitary dual group, also called the Brown algebra, answering the question:
\begin{quote}
\emph{Which distributional symmetry is described by the unitary dual group as the symmetry object?}
\end{quote}

Recall that the Brown algebra \cite{brown, glockner+vonwaldenfels} is the universal complex $*$-algebra $\operatorname{Pol} (U_n^{{\rm nc}})$ generated by elements $u_{jk}$, $j,k\in\{1,\dots,n\}$ such that
\[\sum_{l=1}^n u_{lj}^*u_{lk}=\sum_{l=1}^n u_{jl}u_{kl}^*=\delta_{jk}1,\]
which is equivalent to the matrix $u=(u_{jk})_{1\leq j,k\leq n}$ being unitary, i.e., $u^*u=uu^*=1$. See also~\cite{mcclanahan, voiculescu} for more on the Brown algebra, also called the Brown--Glockner--von Waldenfels algebra. While this algebra does \emph{not} give rise to a compact matrix quantum group (since $u^t=(u_{ji})$ is not invertible \cite[Non-Example 4.1]{wang}), imposing the additional relation $u^t(u^t)^*=(u^t)^*u^t=1$ we obtain the algebra $\operatorname{Pol} (U_n^+)$ of Wang's free unitary quantum group $U_n^+$ \cite{wang}. Now, $U_n^+$ is a~compact matrix quantum group with comultiplication
\[\Delta\colon \ \operatorname{Pol} \big(U_n^+\big)\to \operatorname{Pol} \big(U_n^+\big)\otimes_{\min} \operatorname{Pol} \big(U_n^+\big)\]
to the tensor product, and $U_n^{{\rm nc}}$ is a dual group with a map to the free product:
\[\Delta\colon \ \operatorname{Pol} (U_n^{{\rm nc}})\to \operatorname{Pol} (U_n^{{\rm nc}})\sqcup \operatorname{Pol} (U_n^{{\rm nc}}).\]
Since we have a canonical $*$-homomorphism from $\operatorname{Pol} (U_n^{{\rm nc}})$ to $\operatorname{Pol} (U_n^+)$ mapping generators to generators, one is tempted to view $U_n^+$ as a ``subgroup'' of $U_n^{{\rm nc}}$. One could thus expect a~strengthening of Curran's de Finetti theorem \cite{curran} for $U_n^+$. However, our research reveals a more complex situation.

\subsection{Finite de Finetti theorems for dual group actions}
Firstly, note that while the unitary quantum group $U_n^+$ possesses a Haar state, by a general theorem by Woronowicz, this is \emph{not} the case for the unitary dual group $U_n^{{\rm nc}}$. However, the unitary dual group admits a Haar \emph{trace} as shown by the second author and Ulrich \cite{cebron+ulrich}. Now, due to the special nature of this Haar trace, we may even prove a \emph{finite} de~Finetti theorem for the unitary dual group, in contrast to the situation for $U_n^+$; here, we consider dual group actions, i.e., actions going into the free product of algebras, see Section~\ref{subsec-dga}.\looseness=1

\begin{thm*}[finite de Finetti theorem for dual group actions, \Cref{thm-DeFin1}]
Let $(x_1,\dots,x_n)$ be a finite sequence of random variables in a noncommutative probability space. The following are equivalent:
\begin{enumerate}\itemsep=0pt
\item[$1.$] The family $(x_1, \dots, x_n)$ is composed of $R$-diagonal elements such that the joint free cumulants are zero except those of type $\kappa_{2r} (x_{i_1}^*, x_{i_1}, \dots, x_{i_r}^*, x_{i_r} )$ and $\kappa_{2r}(x_{i_1}, x_{i_2}^*, x_{i_2},\dots, x_{i_r}^*, x_{i_r},\allowbreak x_{i_1}^*)$ for all $r\in \mathbb{N}$. Moreover, these cumulants depend only on the length $2r$. %\label{CumulCond}
\item[$2.$] The family $(x_1, \dots, x_n)$ is invariant under the dual group action of $U_n^{{\rm nc}}$. %\label{InvCond}
\end{enumerate}
\end{thm*}

In case the underlying noncommutative probability space in the above theorem is tracial, the above characterization (1) may be replaced by (see \Cref{prop:realization}):
\begin{itemize}\itemsep=0pt
\item[$(1')$] \emph{The family $(x_1, \dots, x_n)$ has the same $*$-distribution as $(u_1x,\dots,u_nx)$ where $(u_1,\dots,u_n)$ is a freely uniform unit vector, $x$ is self-adjoint and $(u_1,\dots,u_n)$ and $x$ are $*$-free.}
\end{itemize}
And in case we are dealing with a tracial $W^*$-probability space such that $\sum_{i=1}^nx_i^*x_i$ has trivial kernel, the above element $x$ is of the form (see \Cref{cor-finite})
\[x=\sqrt{\sum_{i=1}^nx_i^*x_i}.\]

\subsection{Infinite de Finetti theorems for dual group actions}
Secondly, we may pass to \emph{infinite} de Finetti theorems. The above theorem has a direct analog in the infinite situation replacing $(x_1,\dots,x_n)$ by an infinite sequence $(x_i)_{i\in\mathbb N}$, see \Cref{thm-infinitedual}. However, passing to $W^*$-probability spaces, we obtain a characterization of free centered circular elements, just like in the case of the unitary quantum group $U_n^+$, compare with Curran's result~\cite{curran}.

\begin{thm*}[infinite de Finetti theorem for dual group actions on $W^*$-prob.\ spaces, \Cref{theorem:infinite_setting}]
Let $(x_i)_{i\in\mathbb{N}}$ be an infinite sequence of random variables in some $W^*$-probability space $(M,\varphi)$. The following are equivalent:
\begin{enumerate}\itemsep=0pt
\item[$1.$] There exists some $v\in M$ such that
$(x_i)_{i\in\mathbb{N}}$ is a $\mathcal{B}$-valued free centered circular family whose elements have identical variances
\[
\mathcal{B}\ni b\mapsto \mathbb{E}(x_i b x_i^*)=\varphi(x_ibx_i^*)1_M \qquad\text{and}\qquad \mathcal{B}\ni b\mapsto \mathbb{E}(x_i^*bx_i)=\varphi(b)v,
\]
where $\mathbb E$ is the conditional expectation from $M$ to $W^*(v)$.
\item[$2.$] The distribution of $(x_i)_{i \in \mathbb{N}}$ is invariant under the dual group action of $U_n^{{\rm nc}}$.
\end{enumerate}
In this case, the sequence $\big(\frac{1}{n}\sum_{j=1}^nx_j^*x_j\big)_{n\in \mathbb{N}}$ strongly converges to $v$.
\end{thm*}

In case the underlying $W^*$-probability space is \emph{tracial}, the above characterization (1) may be replaced by (see \Cref{prop:realization_infinite}):
\begin{itemize}\itemsep=0pt
\item[$(1')$] \emph{The sequence $(x_i)_{i \in \mathbb{N}}$ has the same $*$-distribution as $(c_ix)_{i \in \mathbb{N}}$ where $(c_i)_{i \in \mathbb{N}}$ is a sequence of free circular variables, $x$ is self-adjoint and $(c_i)_{i \in \mathbb{N}}$ and $x$ are $*$-free.}
\end{itemize}
See also the version \Cref{cor-infinite} of \Cref{prop:realization_infinite}.

\subsection{Infinite de Finetti theorems for bialgebra actions}
Thirdly, instead of considering actions going into the free product of algebras (dual group actions), we may also consider actions going to the tensor product of algebras (bialgebra actions). In a way, these bialgebra actions are closer to the actions of quantum groups such as $U_n^+$; on the other hand, they are less ``natural'' from the perspective of dual groups. Surprisingly, we have a no-go theorem for bialgebra actions in case we consider $W^*$-probability spaces with faithful states.

\begin{thm*}[no-go de Finetti theorem for bialgebra actions, \Cref{thm-bialg-DeFin}]
The joint $*$-dis\-tri\-bu\-tion of an infinite sequence $(x_i)_{i\in\mathbb{N}}$ of random variables in some $W^*$-probability space is invariant under the $*$-bialgebraic action of $U^{\rm nc}$ if and only if $x_i=0$ for all $i\in\mathbb{N}$.
\end{thm*}

However, if we omit the assumption on the state being faithful, we \emph{do} obtain some de Finetti theorem, at least ``half'' of it, characterizing only one direction.

\begin{proposition*}[half a de Finetti theorem for bialgebra actions, \Cref{prop-bialg-DeFin}]
Let $(x_i)_{i\in\mathbb{N}}$ be an infinite sequence in a $W^*$-probability space $(M,\psi)$ such that $\psi$ is not necessarily faithful. Suppose, there is a $W^*$-subalgebra $\mathbf{1}\in B\subseteq M$ and a conditional expectation $E\colon M\to B$ such that $(x_i)_{i\in\mathbb{N}}$ is a $B$-valued free centered circular family whose elements have identical variances
\[
B\ni b\mapsto \theta(b)=E(x_ibx_i^*)\in B \qquad\text{and}\qquad B\ni b\mapsto \eta(b)=E(x_i^*bx_i)=0
\]
for all $i\in\mathbb{N}$. Then the joint distribution of $(x_i)_{i\in\mathbb{N}}$ is invariant under the $*$-bialgebraic action of $U^{\rm nc}$.
\end{proposition*}

\section{Preliminaries}\label{sec-prel}

\subsection{Partitions of sets}
For any integer $k \geq 1$, let us denote the set $\{1, 2, \dots, k\}$ by $[k]$.

We recall that $\pi = \{V_1, V_2, \dots, V_r\}$ is a \emph{partition} of $[k]$ if and only if the \emph{blocks} $V_i$'s are pairwise disjoint (non empty) subsets of $[k]$ such that $V_1 \cup V_2 \cup \dots \cup V_r = [k]$. Moreover, the partition $\pi$ is called \emph{noncrossing} if for any two distinct blocks $V$ and $W$ of $\pi$ one cannot find four points $1 \leq p < q < r < s \leq k$ such that $\{p,r\} \subset V$ and $\{q,s\} \subset W$.
The set of all noncrossing partitions of $[k]$ is denoted by $\operatorname{NC}(k)$. This is a partially ordered set with the reversed refinement order.

\begin{figure}[h]\centering
\begin{tikzpicture}[scale=0.5]
 \coordinate[label=above:{$1$}] (A1) at (0,0);
 \coordinate[label=above:{$2$}] (A1) at (1,0);
 \coordinate[label=above:{$3$}] (A1) at (2,0);
 \coordinate[label=above:{$4$}] (A1) at (3,0);
	\coordinate[label=above:{$5$}] (A1) at (4,0);
	\draw[blue] (0,0) -- (0,-2) -- (4,-2) -- (4,0);
 \draw[red] (1,0) -- (1,-1) -- (3,-1) -- (3,0);
 \draw[blue] (2,0) -- (2,-2);

	\coordinate[label=above:{$1$}] (A1) at (10,0);
 \coordinate[label=above:{$2$}] (A1) at (11,0);
 \coordinate[label=above:{$3$}] (A1) at (12,0);
 \coordinate[label=above:{$4$}] (A1) at (13,0);
	\coordinate[label=above:{$5$}] (A1) at (14,0);
	\draw (10,0) -- (10,-2) -- (14,-2) -- (14,0);
 \draw (12,0) -- (12,-1) -- (13,-1) -- (13,0);
 \draw (11,0) -- (11,-1);
\end{tikzpicture}
\caption{A crossing partition on the left and a noncrossing one on the right.}
\end{figure}
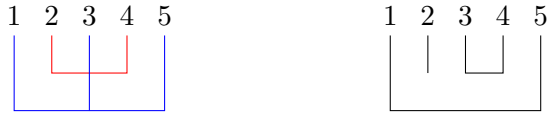

\begin{Definition}[{\cite[Definition~9.14]{nica+speicher}}]
Let $\pi, \sigma \in \operatorname{NC}(k)$ be two noncrossing partitions, we write $\pi \preceq \sigma$ if each block of $\pi$ is contained in one of the blocks of $\sigma$.
\end{Definition}

This partial order induces a lattice structure on $\operatorname{NC}(k)$. The maximal element of $\operatorname{NC}(k)$ with respect to this partial order is the partition consisting of only one block, denoted by $1_k$, and the minimal element is the partition with $k$ blocks, denoted by $0_k$.

\begin{Definition}[{\cite[Definition~9.15]{nica+speicher}}]
The \emph{join} of two partitions $\pi$ and $\sigma$, denoted by $\pi \vee \sigma$, is the minimal element $\tau$ in $\operatorname{NC}(k)$ such that $\pi \preceq \tau$ and $\sigma \preceq \tau$.

The \emph{meet} of two partitions $\pi$ and $\sigma$, denoted by $\pi \wedge \sigma$, is the maximal element $\tau$ in $\operatorname{NC}(k)$ such that $\tau \preceq \pi$ and $\tau \preceq \sigma$.
\end{Definition}

\subsection{Noncommutative probability spaces, cumulants and freeness}

\begin{Definition}[{\cite[Definition~1.12]{mingo+speicher}}]
A \emph{noncommutative probability space} $(A, \varphi)$ consists of a~unital $*$-algebra $A$ and a~state $\varphi \colon A \to \mathbb{C}$, i.e., a unital positive linear functional. An element $a \in A$ is called a \emph{noncommutative random variable}.

If $A$ is a von Neumann algebra and $\varphi$ is a faithful normal state, then $(A, \varphi)$ is called a~\emph{$W^*$-probability space}.
\end{Definition}

Note that we do not assume that $\varphi$ is a trace on $A$.

\begin{Definition}[{\cite[Definition~4.8]{nica+speicher}}]\label{def-jointdistribution}
Let $(A, \varphi)$ be a noncommutative probability space, and $(x_i)_{i \in \mathbb{N}}$ be an infinite sequence of noncommutative random variables in $(A, \varphi)$. Let $\mathcal{Q}_n = \mathbb{C}\langle t_1, t_1^*, \dots, t_n, t_n^*\rangle$ denote the unital $*$-algebra of noncommutative polynomials in $n$ variables with complex coefficients. Then
\begin{align*}
\varphi_x = \varphi_{(x_1, \dots, x_n)} \colon \ \mathcal{Q}_n &\to \mathbb{C},\\
p &\mapsto \varphi\left(p(x)\right)
\end{align*}
is the \emph{joint $*$-distribution} of $x=(x_1,\dots ,x_n)$, where $p \mapsto p(x)$ is the canonical morphism from~$\mathcal{Q}_n$ to~$A$.
\end{Definition}

\begin{Definition}[{\cite[Definition~5.3]{nica+speicher}}]
Let $(A, \varphi)$ be a noncommutative probability space and let $I$ be a fixed index set. Let, for each $i \in I$, $A_i \subset A$ be a unital subalgebra. The subalgebras $(A_i)_{i \in I}$ are called \emph{freely independent} if
\[\varphi(a_1 a_2 \dots a_k) = 0,\]
whenever we have the following:
\begin{itemize}\itemsep=0pt
 \item $k$ is a positive integer,
 \item $a_j \in A_{i(j)}$ for all $1 \leq j \leq k$,
 \item $\varphi(a_j) = 0$ for all $1 \leq j \leq k$,
 \item neighboring elements are from different subalgebras, i.e.,
 \[i(1) \neq i(2) \neq \cdots \neq i(k-1) \neq i(k) .\]
\end{itemize}
If the unital $*$-subalgebras $A_i$ generated by the random variable $a_i$ are freely independent, then we call $(a_i)_{i \in I}$ \emph{$*$-freely independent}, or \emph{$*$-free}.
\end{Definition}

\begin{Definition}[{\cite[Definition~11.3]{nica+speicher}}]
Let $(A, \varphi)$ be a noncommutative probability space. The corresponding \emph{free cumulants} $(\kappa_\pi)_{\pi \in \operatorname{NC}}$ are, for each $n \in \mathbb{N}$, $\pi \in \operatorname{NC}(n)$, multilinear functionals
\begin{align*}
 \kappa_\pi \colon \ A^n & \to \mathbb{C}, \\
 (a_1, \dots, a_n) & \mapsto \kappa_\pi [a_1, \dots, a_n],
\end{align*}
which are defined as follows
\[\kappa_\pi[a_1, \dots, a_n] := \sum_{\substack{\sigma \in \operatorname{NC}(n)\\ \sigma \preceq \pi}} \varphi_\sigma [a_1, \dots, a_n] \mu(\sigma, \pi),\]
where $\mu$ is the Möbius function on the lattice $\operatorname{NC}(n)$ and
\[\varphi_\sigma[a_1, \dots, a_n] = \prod\limits_{V \in \sigma} \varphi_V[a_1, \dots, a_n] := \prod\limits_{\substack{V \in \sigma\\V = \{v_1 < \dots < v_l\}}} \varphi(a_{v_1} \cdots a_{v_l}) .\]
We denote $\kappa_{1_n}$ by $\kappa_n$.
\end{Definition}

\begin{Proposition}[{\cite[Definition~11.4]{nica+speicher}}]
The free cumulants are also determined by the moment cumulant formula:
\[\varphi(a_1 \cdots a_n) = \sum_{\pi \in \operatorname{NC}(n)} \kappa_\pi[a_1, \dots, a_n] .\]
\end{Proposition}
It is possible to compute the free cumulants of products, thanks to the following formula.
\begin{Theorem}[{\cite[Theorem~11.12 and proof of Proposition~11.25]{nica+speicher}}]\label{cum_of_prod}
For all $a_1,\dots, a_{2m}\in A$, we have
\[
\kappa_n(a_1a_2, \dots, a_{2n-1}a_{2n})=\sum_{\substack{\pi\in \operatorname{NC}(2n)\\ \pi\vee\sigma=1_{2n}}}\kappa_\pi[a_1, \dots, a_{2n}] ,
\]
with $\sigma=\{(1,2),\dots,(2n-1,2n)\}$. Moreover, a partition $\pi\in \operatorname{NC}(2n)$ such that $\pi\vee\sigma=1_{2n}$ has the following property:
\[\{(1,2n),(2,3),\dots,(2n-2,2n-1)\} \preceq \pi.\]
\end{Theorem}

\subsection[R-diagonal elements]{$\boldsymbol{R}$-diagonal elements}

\begin{notation}
A tuple $(x_1,\dots ,x_n)$ with entries from a set with two elements $\{a,b\}$ is said to be \emph{alternating $($in $a$ and $b)$}, if $n$ is even and $x_i \neq x_{i+1}$ for all $i=1,\dots ,n-1$.
\end{notation}

\begin{Definition}[{\cite[Definition~15.3]{nica+speicher}}] Let $(A, \varphi)$ be a noncommutative probability space. A~random variable $a \in A$ is called \emph{$R$-diagonal} if for all $n \in \mathbb{N}$ we have that $\kappa_n(a_1, \dots, a_n) =0$ whenever the arguments $a_1, \dots, a_n \in \{a, a^*\}$ are not alternating in $a$ and $a^*$.
\end{Definition}

\begin{Example}
Let us recall that a random variable $c \in A$ is called \emph{circular} when the only non-vanishing cumulants are
\[\kappa_2(c,c^*) = \kappa_2(c^*,c) = 1 .\]
Thus, a circular element is an $R$-diagonal element.
\end{Example}

\begin{Definition}[{\cite[Definition~1.12]{nica+speicher}}]
Let $(A, \varphi)$ be a noncommutative probability space. A~random variable $u \in A$ is called \emph{Haar unitary} if $u$ is a unitary in $A$ and all $*$-moments of the form $\varphi(u^k)$, $k \in \mathbb{Z}$, vanish unless $k = 0$.
\end{Definition}

\begin{Proposition}[{\cite[Proposition~15.1]{nica+speicher}}]
The alternating $*$-cumulants of a Haar unitary $u$ are given by
\[\kappa_{2n}(u, u^*, \dots, u, u^*) = \kappa_{2n}(u^*, u, \dots, u^*, u) = (-1)^{n-1} C_{n-1},\]
where $C_n$ denote the $n$-th Catalan number.
All the other $*$-cumulants of $u$ vanish. Thus any Haar unitary element is an $R$-diagonal element.
\end{Proposition}

\begin{Lemma}[{\cite[Proposition~15.8]{nica+speicher}}]\label{cum_alternating}
Let $\{a_i\}_{i\in I}$ and $\{b_j\}_{j\in J}$ be $*$-free.
We assume that, for all $ m \geq 1$, $(i_1, \dots, i_m) \in I^m,$ and $\underline{e} = (e_1, \dots, e_m) \in \{\varnothing, *\}^m$ the free cumulant
\[
\kappa_{m}\big(a_{i_1}^{e_1},\dots,a_{i_m}^{e_m}\big)
\]
is vanishing whenever $\underline{e}$ is not alternating.

Then, for all $ m \geq 1$, $(i_1, \dots, i_m) \in I^m$, $(j_1, \dots, j_m) \in J^m$ and $\underline{e} = (e_1, \dots, e_m) \in \{\varnothing, *\}^m$ the free cumulant
\[
\kappa_{m}\big((a_{i_1}b_{j_1})^{e_1},\dots,(a_{i_m}b_{j_m})^{e_m}\big)
\]
is vanishing whenever $\underline{e}$ is not alternating.
\end{Lemma}
\begin{proof}In \cite{nica+speicher}, only the case of one $a$ and one $b$ is considered. However, the proof works \emph{mutatis mutandis} with families $\{a_i\}_{i\in I}$ and $\{b_j\}_{j\in J}$.
\end{proof}

\subsection{Operator-valued cumulants}\label{sect-opval}
Let us recall the definition and some basic facts about operator-valued cumulants, see \cite[Chapter~9]{mingo+speicher}.

\begin{Definition}[{\cite[Definitions~9.4 and 9.7]{mingo+speicher}}]
A \emph{conditional expectation} $E\colon A\to B$ from a~unital $*$-algebra $A$ onto a $*$-subalgebra $1\in B\subseteq A$ is defined as a unit-preserving linear map which has the bimodule property $E(b_1ab_2) = b_1E(a)b_2$, for all $b_1,b_2\in B$ and $a\in A$. In such a case we say that $(A,E)$ is a \emph{$B$-valued probability space}. The \emph{$B$-valued cumulants} of $E$ are defined implicitly via the formula
\[
E(a_1\cdots a_n) = \sum_{\pi\in \operatorname{NC}(n)} \kappa_{\pi}^E [a_1, \dots , a_n].
\]
Here the arguments of $\kappa_{\pi}^E$ are distributed according to the blocks of $\pi$, and the cumulants inside~$\kappa_{\pi}^E$ are nested according to the nesting of the block of $\pi$, see \cite[Chapter~9]{mingo+speicher} for details and examples. We denote $\kappa^E_{1_n}$ by $\kappa^E_{n}$.
\end{Definition}

Note that the bimodule property for $E$ implies that $\kappa^E_{n}$ is a map on the $B$-module tensor product $A\otimes_{B} A \otimes_{B} \cdots \otimes_{B} A$.

\begin{Example}The cumulants $\big(\kappa^E_{n}\big)_{n\ge 1}$ of a $B$-valued centered circular element $c$ are of the form
\begin{gather*}
\kappa^E_{n}\big(b_0c^{e_1}b_1, c^{e_2}b_2,\dots, c^{e_n}b_n\big) =
\begin{cases}
b_0\eta(b_1)b_2 & \mbox{if }n=2\mbox{ and } (e_1,e_2)=(*,\varnothing), \\
b_0\theta(b_1)b_2 &\mbox{if }n=2\mbox{ and } (e_1,e_2)=(\varnothing,*), \\
0 & \mbox{else}
\end{cases}
\end{gather*}
with $\eta(b)= \kappa_{2}^E(c^*b, c)$ and $\theta(b)=\kappa_{2}^E(cb, c^*)$.
The cumulants $\kappa_{\pi}^E \big[b_0x_1^{e_1}b_1, \dots , x_1^{e_k}b_k\big]$ appearing in the sum above are uniquely determined by $\eta(b)= \kappa_{2}^E(x_1^*b, x_1)$ and $\theta(b)=\kappa_{2}^E(x_1b, x_1^*)$.

More generally, we said that a sequence $(x_i)_{i\in\mathbb{N}}$ is a $B$-valued free circular family with common covariance if their cumulants are of the form
\begin{gather*}
\kappa^E_{n}\big(b_0x_{i(1)}^{e_1}b_1, x_{i(2)}^{e_2}b_2,\dots, x_{i(n)}^{e_n}b_n\big) =
\begin{cases}
b_0\eta(b_1)b_2\! & \text{if} \ n=2, \ i(1)=i(2)\ \text{and} \ (e_1,e_2)=(*,\varnothing), \\
b_0\theta(b_1)b_2\! & \text{if} \ n=2, \ i(1)=i(2) \ \text{and} \ (e_1,e_2)=(\varnothing,*), \\
0 & \text{else},
\end{cases}\!
\end{gather*}
with $\eta(b)= \kappa_{2}^E(x_1^*b, x_1)$ and $\theta(b)=\kappa_{2}^E(x_1b, x_1^*)$.
Note that a sequence $(x_i)_{i\in\mathbb{N}}$ is a $B$-valued free circular family with common covariance if and only if
\[
E\big(b_0x_{i_1}^{e_1}b_1\cdots x_{i_k}^{e_k}\big) =
 \begin{cases}
\displaystyle \sum_{\pi\in \operatorname{NC}^e_2(k), \,\pi\preceq \operatorname{ker} i} \kappa_{\pi}^E \big[b_0x_1^{e_1}b_1, \dots , x_1^{e_k}b_k\big] & \text{if $k$ even}, \\
0 & \text{if $k$ odd},
\end{cases}
\]
where
$
\operatorname{NC}^e_2(k) = \{\pi\in \operatorname{NC}_2(k);\forall \{s,t\}\in\pi, e_s\not=e_t\}
$
and $\operatorname{ker} i$ is the partition obtained by forming blocks consisting in equal indices in $i=(i_1,\dots,i_k)$.
\end{Example}

\subsection{Conditional expectations of free algebras}
\label{sec-ce}

Given two unital $*$-algebras $A$ and $B$, the free product $A\sqcup B$ is the unique unital $*$-algebra with $*$-homomorphisms $i_A\colon A\to A\sqcup B$ and $i_B\colon B\to A\sqcup B$ such that, for all $*$-homomorphisms $f\colon A\to C$ and $g\colon B\to C$, there exists a unique $*$-homomorphism $f\sqcup g\colon A\sqcup B\to C$ such that $f=(f\sqcup g)\circ i_A$ and $g=(f\sqcup g)\circ i_B$.

Let $(A_1, \varphi_1)$ and $(A_2, \varphi_2)$ be two unital $*$-algebras endowed with a state and consider the unital $*$-algebra $A_1 \sqcup A_2$ with the state $\varphi = \varphi_1 \ast \varphi_2$. Following \cite[Exercise~18]{mingo+speicher}, let us define a conditional expectation from $A_1 \sqcup A_2$ to $A_1$. By setting $A^{o}_i=A_i\cap\ker \varphi_i$, we have the decomposition
\[
A_1 \sqcup A_2= \mathbb{C}1\oplus \bigoplus_{n=1}^\infty \bigoplus_{i_1\neq i_2 \neq \cdots \neq i_n \in \{1,2\}} A_{i_1}^o \otimes \dots \otimes A_{i_n}^o.
\]
Let us define the linear map $E^\varphi_{\varphi_1} \colon A_1 \sqcup A_2 \to A_1$ to be the identity on $A_1=\mathbb{C}1\oplus A^{o}_1$ and zero on all remaining summands. Similarly, let us define $E^\varphi_{\varphi_2} \colon A_1 \sqcup A_2 \to A_2$ to be the identity on $A_2=\mathbb{C}1\oplus A^{o}_2$ and zero on all remaining summands.

\begin{Proposition}
The linear maps $E^\varphi_{\varphi_i} \colon A_1 \sqcup A_2 \to A_i$ are two conditional expectations preserving $\varphi$, in the sense that $\varphi \circ E^\varphi_{\varphi_i} =\varphi$.
\end{Proposition}
\begin{proof}
The bimodule property is a direct consequence of the definition. We have $\varphi(a) = \varphi \circ E^\varphi_{\varphi_i}(a)$ if $a\in A_i=\mathbb{C}1\oplus A^{o}_i$ (because $E^\varphi_{\varphi_i}(a)=a$) and we have $\varphi(a)=0=\varphi \circ E^\varphi_{\varphi_i}(a)$ for $a$ in any other of the summands, because of freeness of $A_1$ from $A_2$.
\end{proof}
\begin{Remark}
As in \cite[Theorem~19]{mingo+speicher}, it is possible to prove general formulas for calculating such conditional expectations:
\begin{gather*}
\forall p \geq 1,\quad \forall a_1, \dots, a_p \in A_1,\quad \forall b_1, \dots, b_p \in A_2,\\
E^\varphi_{\varphi_2}[a_1 b_1 \cdots a_p b_p] = \sum_{\pi \in \operatorname{NC}(p)} \kappa_{\pi}^{\varphi_1}[a_1, \dots, a_p] \prod_{\substack{V \in K(\pi)\\ V \neq V_{\rm last}}} (\varphi_2)_V[b_1, \dots, b_p] \prod_{v \in V_{\rm last}}^{_\to} b_v,\\
E^\varphi_{\varphi_1}[b_1 a_1 \cdots b_p a_p] = \sum_{\pi \in \operatorname{NC}(p)} \kappa_{\pi}^{\varphi_2}[b_1, \dots, b_p] \prod_{\substack{V \in K(\pi)\\ V \neq V_{\rm last}}} (\varphi_1)_V[a_1, \dots, a_p] \prod_{v \in V_{\rm last}}^{_\to} a_v,
\end{gather*}
where $K(\pi)$ denotes the Kreweras complement of $\pi$ (see \cite[Definition~9.21]{nica+speicher}), $V_{\rm last}$ is the block of the noncrossing partition $K(\pi)$ containing the uttermost right point and $\prod\limits_{v \in V_{\rm last}}^{_\to}\hspace{-5pt}x_v$ is the noncommutative product where the $v$'s are taken in increasing order.
\end{Remark}

In particular, we have:
\begin{Corollary}
Let $n\in \mathbb{N}$ and $x=(x_1, \dots, x_n)$ be a family of random variables in a noncommutative probability space $(A, \varphi)$. On the free product of the Brown algebra $\operatorname{Pol}(U_n^{\rm nc})$ and $\mathcal Q_n$ we have the conditional expectation $E^{h_n*\varphi_x}_{h_n}$. Thus $E^{h_n*\varphi_x}_{h_n}\circ\alpha_n$ is a map from $\mathcal Q_n$ to $\operatorname{Pol}(U_n^{\rm nc})$.
\end{Corollary}

We have also nice formulas for the operator-valued cumulants $\kappa_{m}^{E}$ with $E=E^\varphi_{\varphi_i}$, given by the following theorem.

\begin{Theorem}[{\cite[Theorem~3.6]{nicaspeichershlyakhtenko}}]\label{cum_of_cond}
Let $A$ and $B$ be two free subalgebras of a noncommutative probability space $(M,\varphi)$. We assume that there exists a $\varphi$-preserving conditional expectation $E\colon M\to B$ and that $\left.\varphi\right|_{B}$ is non-degenerate $($in the sense that $0$ is the unique $b_1\in B$ such that $\varphi(b_1b_2)=0$ for all $b_2\in B)$. Then, for all $m\geq 1$ and all $a_1,\dots,a_m\in A$, $b_0,b_1,\dots,b_{m}\in B$, we have
\[
\kappa_{m}^{E}(b_0a_1b_1, \dots , a_{m-1}b_{m-1}, a_mb_m)=\kappa^{\varphi}_m(a_1,\dots,a_m)\varphi(b_1)\cdots\varphi(b_{m-1})b_0b_m.
\]
\end{Theorem}

\begin{Corollary} \label{cor::cum_of_cond}
Let $(A, \varphi_1)$ and $(B, \varphi_2)$ be two noncommutative probability spaces and consider the conditional expectation $E:=E^\varphi_{\varphi_2} \colon A \sqcup B \to B$. Then, for all $m\geq 1$ and all $a_1,\dots,a_m\in A$, $b_0,b_1,\dots,b_{m}\in B$, we have
\[
\kappa_{m}^{E}(b_0a_1b_1, \dots , a_{m-1}b_{m-1}, a_mb_m)=\kappa^{\varphi_1}_m(a_1,\dots,a_m)\varphi_2(b_1)\cdots\varphi_2(b_{m-1})b_0b_m.
\]
\end{Corollary}

\subsection{Two structures on the Brown algebra}

\begin{Definition}[\cite{brown, glockner+vonwaldenfels}]
Let $n\ge 1$. Denote by $\operatorname{Pol}(U_n^{\rm nc})$ the universal unital $*$-algebra with~$n^2$ generators~$u_{jk}$, $1\le j$, $k\le n$ and the relations
\[
\sum_{\ell=1}^n u_{j\ell}u^*_{k\ell} = \delta_{jk}1=\sum_{\ell=1}^n u^*_{\ell j} u_{\ell k}.
\]
This algebra is called the Brown algebra or Brown--Glockner--von Waldenfels algebra.
\end{Definition}

For $*$-homomorphisms $f\colon A\to C$ and $g\colon B\to D$, we denote by $f\bsqcup g$ the $*$-homomorphism $(i_C\circ f)\sqcup(i_D\circ g)\colon A\sqcup B\to C\sqcup D$, whereas, as above, $f\sqcup g\colon A\sqcup B\to C$ denotes the unique $*$-homomorphism such that $f=(f\sqcup g)\circ i_A$ and $g=(f\sqcup g)\circ i_B$, in case $f\colon A\to C$ and $g\colon B\to C$ are given. Recall that $i_A$ denotes $i_A\colon A\to A\sqcup B$.

\begin{Definition}[\cite{voiculescu}]
A dual group in the sense of Voiculescu is composed of a unital $*$-algebra~$A$ and three unital $*$-homomorphisms $\Delta \colon A \to A \sqcup A$, $\delta \colon A \to \mathbb{C}$ and $\Sigma \colon A \to A$ such that
\begin{itemize}\itemsep=0pt
\item $\Delta$ is a coassociative coproduct, i.e., $(\mathrm{id}_A \bsqcup \Delta)\circ \Delta = (\Delta \bsqcup \mathrm{id}_A)\circ \Delta$,
\item $\delta$ is a counit, i.e., $(\delta \bsqcup \mathrm{id}_A)\circ \Delta = \mathrm{id}_A = (\mathrm{id}_A \bsqcup \delta)\circ \Delta$,
\item $\Sigma$ is a coinverse, i.e., $(\Sigma \sqcup \mathrm{id}_A)\circ \Delta = \delta(\cdot) 1_A = (\mathrm{id}_A \sqcup \Sigma)\circ \Delta$.
\end{itemize}
\end{Definition}

\begin{Lemma}[\cite{voiculescu}]
The Brown algebra $\operatorname{Pol}(U_n^{\rm nc})$ is a dual group when it is endowed with the following $*$-homomorphisms:
\begin{itemize}\itemsep=0pt
\item the coproduct $\Delta$ defined on the generators by $\Delta(u_{ij}) = \sum_{k = 1}^n u_{ik}^{(1)} u_{kj}^{(2)}$,
\item the counit $\delta$ given on the generators by $\delta(u_{ij}) = \delta_{ij}$,
\item the coinverse $\Sigma$ given by $\Sigma(u_{ij}) = u_{ji}^*$.
\end{itemize}
\end{Lemma}

\begin{Remark}
The Hopf $*$-algebra $\operatorname{Pol}(U_n^+)$ of the quantum unitary group $U_n^+$ defined by Wang~\cite{wang} is obtained by dividing $\operatorname{Pol}(U_n^{\rm nc})$ by the ideal generated by the relations
\[
\sum_{\ell=1}^n u_{\ell j}u^*_{\ell k} = \delta_{jk}1=\sum_{\ell=1}^n u^*_{j\ell} u_{k\ell}.
\]
In order to distinguish the two algebras $\operatorname{Pol}(U_n^{\rm nc})$ and $\operatorname{Pol}(U_n^+)$, we will denote the generators of the former by $u_{jk}$, and the generators of the latter by $w_{jk}$, $1\le j$, $k\le n$. Denote the canonical quotient map by $\pi \colon \operatorname{Pol}(U_n^{\rm nc})\to\operatorname{Pol}(U_n^+)$, $\pi(u_{jk})=w_{jk}$; this map can be viewed as the restriction homomorphism of the inclusion $U_n^+ \subseteq U_n^{\rm nc}$ in some sense. But note that $U_n^{\rm nc}$ is not a quantum group.
\end{Remark}

Note that quantum groups always have a Haar state, but this is not true for dual groups \cite{cebron+ulrich}. The second author and Ulrich also define a weaker notion: the Haar trace, and prove that the Brown algebra admits a Haar trace with respect to the free product of states (this is also the state McClanahan used in \cite{mcclanahan}).

Let us recall the description of the cumulants of the $u_{ij}$'s with respect to the free Haar trace in \cite{cebron+ulrich}.
\begin{Proposition}[{\cite[Corollary~2.8]{cebron+ulrich}}] \label{cor_CU15}
The free cumulants of the noncommutative random variables $(u_{ij})_{1\leq i,j\leq n}$ and $((u^*)_{ij})_{1\leq i,j\leq n}:=(u_{ji}^*)_{1\leq i,j\leq n}$ in $(\operatorname{Pol}(U_n^{\rm nc}), h_n)$ are given as follows. Let
\[
1\leq i_1,k_1,\dots,i_m,k_m \leq n \qquad\text{and}\qquad (e_1,\dots,e_m)\in \{\varnothing,*\}^m.
\]
If the indices are cyclic, i.e., $i_1=k_m$ and $i_j=k_{j-1}$ for $2\leq j\leq m$, $m$ is even and the $e_j$ are alternating, we have
\[
\kappa_m^{h_n} \big[\big(u^{e_1}\big)_{i_1k_1},\dots,\big(u^{e_m}\big)_{i_mk_m}\big] = n^{1-m}(-1)^{m/2-1}C_{m/2-1},
\]
where, as before, $C_n$ denotes the $n$-th Catalan number.
Otherwise, the left-hand side is equal to zero.
\end{Proposition}

We will denote the reduced and universal $C^*$-algebras and the von Neumann algebras associated to $U_n^+$ and $U_n^{\rm nc}$ by $C_r(U_n^{+/{\rm nc}})$, $C_u(U_n^{+/{\rm nc}})$, and $VN(U_n^{+/{\rm nc}})$. Here the reduced $C^*$-algebra of $U_n^{\rm nc}$ is defined as the closure of the image of $\operatorname{Pol}(U_n^{\rm nc})$ under the GNS-representation w.r.t.\ to the Haar trace defined in \cite{cebron+ulrich}, and we set $VN(U_n^{\rm nc})=C_r(U_n^{\rm nc})''$.

\subsection{Dual group actions}\label{subsec-dga}

Note that the Brown algebra $\operatorname{Pol}(U_n^{\rm nc})$ is both an involutive bialgebra (but not a Hopf $*$-algebra for $n>1$), cf.\ \cite{glockner+vonwaldenfels}, and a dual group (in the sense of Voiculescu \cite{voiculescu}), so we can consider two kinds of actions.\footnote{Terminology: an \emph{action} of $U_n^{\rm nc}$ or $U_n^+$ is a \emph{coaction} of (one of) their algebras $\operatorname{Pol}(U^{{\rm nc}/+}_n)$, $C(U^{{\rm nc}/+}_n)$.}

Note that a coaction of $\operatorname{Pol}(U_n^{\rm nc})$ as a bialgebra induces also a coaction of $\operatorname{Pol}(U_n^+)$, so we can exploit the results of \cite{curran}, see \Cref{sec-bialg}.

\begin{Definition}\label{finitedualaction}
An action of the dual group $G = (\operatorname{Pol}(G),\Delta,\delta,\Sigma)$ on the unital $*$-algebra $M$ is a morphism $\alpha \colon M \to \operatorname{Pol}(G) \sqcup M$ satisfying
\[(\Delta \text{\b{$\sqcup$}} {\rm id}_M)\circ \alpha = ({\rm id}_{\operatorname{Pol}(G)} \text{\b{$\sqcup$}} \alpha)\circ \alpha \qquad\text{and}\qquad (\delta \text{\b{$\sqcup$}} {\rm id}_M) \circ \alpha = {\rm id}_M .\]
\end{Definition}

Define $\mathcal{Q}_n$ as the $*$-algebra of noncommutative polynomials with $n$ variables and complex coefficients. Define $\alpha_n \colon \mathcal{Q}_n \to \operatorname{Pol}(U^{\rm nc}_n) \sqcup \mathcal{Q}_n$ as the unital $*$-homomorphism satisfying
\[\alpha_n(t_i) = \sum_{j = 1}^{n} u_{ij}t_j\]
and the corresponding fixed point algebra
\[\mathcal{Q}_n^{\rm fix} := \{ p \in \mathcal{Q}_n \mid \alpha_n(p)=p\} .\]
It is straightforward to prove that $\alpha_n$ is a dual action of $\operatorname{Pol}(U^{\rm nc}_n)$ on $\mathcal{Q}_n$, and thus we have the following lemma.

\begin{Lemma}
The Brown algebra acts as a dual group on the algebra $\mathcal{Q}_n$ of noncommutative polynomials with $n$ variables and complex coefficients.
\end{Lemma}

\begin{Proposition} \label{mainthm}
We have
\[
\mathcal{Q}_n^{\rm fix}=\mathbb{C}\left\langle \sum_{j=1}^n t_j^* t_j \right\rangle
\]
and more generally, if $\varphi_n$ is a non-degenerate state on $\mathcal{Q}_n$, we have
\[ E_{\varphi_n}^{h_n * \varphi_n} \circ \alpha_n (\mathcal{Q}_n) = \mathcal{Q}_n^{\rm fix}=\mathbb{C}\left\langle \sum_{j=1}^n t_j^* t_j \right\rangle .\]
\end{Proposition}

\begin{proof}
``$\mathbb{C}\big\langle \sum_{j=1}^n t_j^* t_j \big\rangle\subseteq \mathcal{Q}_n^{\rm fix}$'':
It follows from the relations of $\operatorname{Pol}(U^{\rm nc}_n)$ that
\[ \alpha_n\Bigg(\sum_{j=1}^n t_j^* t_j\Bigg) = \sum_{j=1}^n \sum_{k=1}^n \sum_{l=1}^n t_k^* u_{jk}^* u_{jl} t_l = \sum_{k=1}^n \sum_{l=1}^n t_k^* \delta_{kl} t_l = \sum_{k=1}^n t_k^* t_k .\]
Hence we have $\sum_{j=1}^n t_j^* t_j \in \mathcal{Q}_n^{\rm fix}$.

``$\mathcal{Q}_n^{\rm fix}\subseteq E_{\varphi_n}^{h_n * \varphi_n} \circ \alpha_n (\mathcal{Q}_n)$'':
Let $x\in \mathcal{Q}_n^{\rm fix}$. Then we have
\[ E_{\varphi_n}^{h_n * \varphi_n} \circ \alpha_n (x)
= E_{\varphi_n}^{h_n * \varphi_n} (x)
= \kappa_{1}^{h_n}[1] \cdot x
= x .\]

``$E_{\varphi_n}^{h_n * \varphi_n} \circ \alpha_n (\mathcal{Q}_n) \subseteq \mathbb{C}\big\langle \sum_{j=1}^n t_j^* t_j \big\rangle$'':
We want to compute the moments of $(\alpha_n(t_1),\dots,\alpha_n(t_n))$ with respect to $E_{\varphi_n}^{h_n \ast \varphi_n}$. In order to do so, we will compute their free cumulants.
We set $E:= E_{\varphi_n}^{h_n \ast \varphi_n}$.
Let $ m \geq 1$, $(i_1, \dots, i_m) \in [n]^m$, $\underline{e} = (e_1, \dots, e_m) \in \{\varnothing, *\}^m$ and $b_1, \dots,b_{m-1} \in \mathcal{Q}_n^{\rm fix}$.
We can use \Cref{cor::cum_of_cond} in order to compute
\[\kappa_{m}^{E} \big(\alpha\big(t_{i_1}^{e_1}\big)b_1, \dots , b_{m-1} \alpha\big(t_{i_m}^{e_m}\big)\big)=\sum_{k_1, \dots, k_m \in [n]^m} \kappa_m^{h_n} \big(u_{i_1k_1}^{e_1},\dots, u_{i_mk_m}^{e_m}\big)\times (*).\]
Thanks to the data of the free cumulants of the $(u_{ij})$ (given by \Cref{cor_CU15}), we see that all the terms of the sum are vanishing if $\underline{e}$ is not alternating.

Let us examine the case where $m=2r$ and $\underline{e}=(\varnothing,*,\dots,\varnothing,*)$ is alternating,
\begin{align*}
 &\kappa_{m}^{E} \big(\alpha\big(t_{i_1}^{e_1}\big)b_1, \dots , b_{m-1} \alpha\big(t_{i_m}^{e_m}\big)\big)\\
 &\qquad{}=\sum_{k_1, \dots, k_{2r} \in [n]^{2r}} \kappa_m^{h_n} \big(u_{i_1k_1},\dots, u_{i_mk_m}^*\big) \varphi\big(t_{k_1}b_1t_{k_2}^*\big) \varphi(b_2) \cdots \varphi\big(t_{k_{2r-1}}b_{2r-1}t_{k_{2r}}^*\big)\\
 &\qquad{}=(-1)^{r-1}n^{1-2r}C_{r-1}\delta_{i_1=i_{2r},i_2=i_3, \dots} \\
 &\qquad\quad{}\times \sum_{k_1=k_{2},k_3=k_4, \ldots \in [n]^{2r}} \varphi\big(t_{k_1}b_1t_{k_2}^*\big) \varphi(b_2)\cdots \varphi\big(t_{k_{2r-1}}b_{2r-1}t_{k_{2r}}^*\big).
\end{align*}
Similarly, if $m=2r$ and $\underline{e}=(*,\varnothing,\dots,*,\varnothing)$ is alternating, we have
\begin{gather*}
\kappa_{m}^{E} \big(\alpha\big(t_{i_1}^{e_1}\big)b_1, \dots , b_{m-1} \alpha\big(t_{i_m}^{e_m}\big)\big)\\
\quad =(-1)^{r-1}n^{1-2r}C_{r-1}\delta_{i_1=i_2,i_3=i_4, \dots}\!
 \sum_{k_1=k_{2r},k_2=k_3, \ldots \in [n]^{2r}} t_{k_1}^*t_{k_{2r}} \varphi(b_1) \varphi\big(t_{k_2}b_2t_{k_3}^*\big)\cdots
\varphi(b_{2r-1}).\end{gather*}
Finally, we have shown that the maps
\[(b_1,\dots,b_{m-1})\mapsto \kappa_{m}^{E} \big(\alpha\big(t_{i_1}^{e_1}\big)b_1, \dots , b_{m-1} \alpha\big(t_{i_m}^{e_m}\big)\big),\]
leave $\mathbb{C}\big\langle \sum_{j=1}^n t_j^* t_j \big\rangle$ invariant. This means that the moments of $(\alpha_n(t_1),\dots,\alpha_n(t_n))$ w.r.t.\ $E=E_{\varphi_n}^{h_n \ast \varphi_n}$ are in $\mathbb{C}\big\langle \sum_{j=1}^n t_j^* t_j \big\rangle$, which implies that $E_{\phi_n}^{h_n * \varphi_n} \circ \alpha_n$ takes values in $\mathbb{C}\big\langle \sum_{j=1}^n t_j^* t_j \big\rangle$.
\end{proof}

\section{Finite de Finetti theorems for dual group actions}\label{sec-dg}

In this section, we first consider the case of finite sequences: surprisingly, unlike in the quantum group situation, we are able to
prove a \emph{finite} de Finetti theorem for the dual group action of~$\operatorname{Pol}(U_n^{\rm nc})$. Also, when restricting to the tracial case, we will give a refined characterization in terms of freely uniform unit vectors.

\subsection{Invariance for finite sequences in the general case}

\begin{Definition}\label{def-inv-dg-finite}
The $n$-tuple $(x_1, \dots, x_n)$ of random variables in a noncommutative probability space $(A, \varphi)$ is called \emph{invariant under $U_n^{\mathrm{nc}}$}, whenever its distribution is invariant under the action~$\alpha_n$, i.e., $\varphi_x 1_{U_n^{nc}} = E_{h_n}^{h_n \ast \varphi_x} \circ \alpha_n$ or more precisely
\begin{gather}
\forall m \geq 1, \ \forall (i_1, \dots, i_m) \in [n]^m, \ \forall \underline{e} = (e_1, \dots, e_m) \in \{\varnothing, *\}^m,\nonumber\\
\varphi\big(x_{i_1}^{e_1} \cdots x_{i_m}^{e_m}\big) 1_{U_n^{nc}} = E_{h_n}^{h_n \ast \varphi_x} \circ \alpha_n\big(t_{i_1}^{e_1} \cdots t_{i_m}^{e_m}\big).
\label{nInv}
\end{gather}
\end{Definition}

Here comes our finite de Finetti theorem for the dual group action of $\operatorname{Pol}(U_n^{\rm nc})$.

\begin{Theorem}\label{thm-DeFin1}
Let $n$ be a natural integer and $(x_1, \dots, x_n)$ be a family of random variables in a~noncommutative probability space $(A, \varphi)$. The following are equivalent:
\begin{enumerate}\itemsep=0pt
\item[$1.$] The family $(x_1, \dots, x_n)$ is composed of $R$-diagonal elements such that the joint free cumulants are zero except those of type $\kappa_{2r}(x_{i_1}^*, x_{i_1}, \dots, x_{i_r}^*, x_{i_r})$ and $\kappa_{2r}(x_{i_1}, x_{i_2}^*, x_{i_2},\dots, x_{i_r}^*, x_{i_r}, \allowbreak x_{i_1}^*)$ for all $r\in \mathbb{N}$. Moreover, these cumulants depend only on the length~$2r$. %\label{CumulCond}
\item[$2.$] The family $(x_1, \dots, x_n)$ is invariant under the dual group action $\alpha_n$.%\label{InvCond}
\item[$3.$] We have $\varphi_x=(h_n*\varphi_x)\circ\alpha_n$.
\end{enumerate}
\end{Theorem}

\begin{proof}
Implication $(3)\rightarrow (1)$:
Let $ m \geq 1$, $\underline{i} = (i_1, \dots, i_m) \in [n]^m$, $\underline{e} = (e_1, \dots, e_m) \in \{\varnothing, *\}^m$. We compute
\begin{align*}
\kappa_m^{\varphi}\big(x_{i_1}^{e_1},\dots , x_{i_m}^{e_m}\big)
 &=\kappa_m^{\varphi_x}\big(t_{i_1}^{e_1},\dots , t_{i_m}^{e_m}\big)
 =\kappa_m^{h_n*\varphi_x}\big(\alpha_n\big(t_{i_1}^{e_1}\big),\dots , \alpha_n\big(t_{i_m}^{e_m}\big)\big)\\
 &=\sum_{k_1, \dots, k_{m} \in [n]} \kappa_m^{h_n*\varphi_x}\big(\big(u_{i_1k_1}t_{k_1}\big)^{e_1},\dots , \big(u_{i_mk_m}t_{k_m}\big)^{e_m}\big),\end{align*}
 where we used $\varphi_x=(h_n*\varphi_x)\circ\alpha_n$ for the second equality.
 By \Cref{cor_CU15}, the free cumulants $\kappa_m^{h_n}\big(u_{i_1k_1}^{e_1},\dots , u_{i_mk_m}^{e_m}\big)$ vanish if $\underline{e}$ is not alternating. Moreover, the elements $\{u_{ik}\}$ and $\{t_k\}$ are by construction $*$-free with respect to $(\operatorname{Pol}(U_n^{\rm nc}) \sqcup \mathcal Q_n ,h_n*\varphi_x)$. Thus \Cref{cum_alternating} implies that all the terms in the above sum are vanishing if $\underline{e}$ is not alternating.

 Let us examine the case where $m=2r$ and $\underline{e}=(\varnothing,*,\dots,\varnothing,*)$ is alternating, the case $\underline{e}=(*,\varnothing,\dots,*,\varnothing)$ being similar.
 By using \Cref{cum_of_prod}, we have
\begin{align*}
 \kappa_m^{\varphi}(x_{i_1},\dots , x_{i_m}^*)
 &=\sum_{k_1, \dots, k_{m} \in [n]}\kappa_m^{h_n*\varphi_x}(u_{i_1k_1}t_{k_1},\dots , t_{k_m}^*u_{i_mk_m}^*)\\
 &=\sum_{k_1, \dots, k_{m} \in [n]}\sum_{\substack{\pi\in \operatorname{NC}(2m)\\ \pi\vee \sigma=1_{2m}}}\kappa_\pi^{h_n*\varphi_x}(u_{i_1k_1},t_{k_1},\dots , t_{k_m}^*,u_{i_mk_m}^*),
\end{align*}
with $\sigma=\left\{\{1,2\}, \dots, \{2m-1, 2m\}\right\}$.

Since the elements $\{u_{ik}\}$ and $\{t_k\}$ are $*$-free, their mixed cumulants vanish and hence only such partitions $\pi\in \operatorname{NC}(2m)$ contribute to the above sum for which each of their blocks connects elements only from $\{u_{ik},u_{ik}^*\}$ or only from $\{t_{k},t_{k}^*\}$. For such a partition $\pi\in \operatorname{NC}(2m)$, we denote by $\pi_u\in \operatorname{NC}(m)$ the subpartition of $\pi$ corresponding to the elements $\{u_{ik},u_{ik}^*\}$ and by $\pi_t\in \operatorname{NC}(m)$ the subpartition of $\pi$ corresponding to the elements $\{t_{k},t_{k}^*\}$.

Then the non-zero mixed cumulants can be written as
\[\kappa_\pi^{h_n*\varphi_x}(u_{i_1k_1},t_{k_1},\dots , t_{k_m}^*,u_{i_mk_m}^*) = \kappa_{\pi_u}^{h_n}(u_{i_1k_1}, \dots, u_{i_mk_m}^*) \kappa_{\pi_t}^{\varphi_x}(t_{k_1}, \dots, t_{k_m}^*) .\]
Moreover, \Cref{cum_of_prod} tells us that
\[\{(1,2m),(2,3),\dots,(2m-2,2m-1)\} \preceq \pi\]
in all the terms of the above sum. Thus, thanks to the data of the free cumulants of the $(u_{ij})$ given by \Cref{cor_CU15}, all the terms $\kappa_{\pi_u}^{h_n}(u_{i_1k_1}, \dots, u_{i_mk_m}^*)$ are vanishing unless $i_1= i_{m}$, $i_2= i_{3}$, $\dots$ and $k_1 = k_2$, $k_3=k_4$, $\dots$. Moreover, in this particular case, the value of $\kappa_{\pi_u}^{h_n}(u_{i_1k_1}, \dots, u_{i_mk_m}^*)$ does not depend on the indices $i_1,\dots ,i_l$.

Implication $(1)\rightarrow (2)$:
We set $E:= E_{h_n}^{h_n \ast \varphi_x}$. Let $ m \geq 1$, $(i_1, \dots, i_m) \in [n]^m$, $\underline{e} = (e_1, \dots, e_m) \in \{\varnothing, *\}^m$.
We can use \Cref{cor::cum_of_cond} in order to compute
\begin{align*}
\kappa_{m}^{E} \big(\alpha_n\big(t_{i_1}^{e_1}\big), \dots , \alpha_n\big(t_{i_m}^{e_m}\big)\big)
&=\sum_{k_1, \dots, k_m \in [n]^m} \kappa_{m}^{E} \big(\big(u_{i_1k_1}t_{k_1}\big)^{e_1}, \dots , \big(u_{i_mk_m}t_{k_m}\big)^{e_m}\big)\\
&=\sum_{k_1, \dots, k_m \in [n]^m} \kappa_m^{\varphi_x} \big(t_{k_1}^{e_1},\dots, t_{k_m}^{e_m}\big)\times (*) \\
&=\sum_{k_1, \dots, k_m \in [n]^m} \kappa_m^{\varphi} \big(x_{k_1}^{e_1},\dots, x_{k_m}^{e_m}\big)\times (*),
\end{align*}
where the term $(*)$ depends on $\underline{i}$, $\underline{e}$, and $k_1, \dots, k_m$. By assumption the cumulants $\kappa_m^{\varphi} \big(x_{k_1}^{e_1},\dots,\allowbreak x_{k_m}^{e_m}\big)$ vanish if $\underline{e}$ is not alternating and the same holds for $\kappa_{m}^{E} \big(\alpha_n(t_{i_1}^{e_1}), \dots , \alpha_n(t_{i_m}^{e_m})\big)$ by the above equation.

Let us examine the case where $m=2r$ and $\underline{e}=(\varnothing,*,\dots,\varnothing,*)$ is alternating.
\begin{gather*}
\kappa_{m}^{E} (\alpha(t_{i_1}), \dots , \alpha(t_{i_m}^*))\\
\qquad{} =\sum_{k_1, \dots, k_{2r} \in [n]^{2r}} \kappa_m^{\varphi_x} (t_{k_1},\dots, t_{k_m}^*)
 h_n(u_{i_2k_2}^*u_{i_3k_3})\cdots h_n(u_{i_{2r-2}k_{2r-2}}^*u_{i_{2r-1}k_{2r-1}})u_{i_1k_1}u^*_{i_{2r}k_{2r}}\\
 \qquad{}=\kappa_m^{\varphi_x} (t_{1},\dots, t_{1}^*) \sum_{\substack{k_1=k_{2r}, \\ k_2=k_3, \ldots \in [n]^{2r}}} h_n(u_{i_3k_3}u_{i_2k_2}^*)\cdots h_n(u_{i_{2r-1}k_{2r-1}}u_{i_{2r-2}k_{2r-2}}^*)u_{i_1k_1}u^*_{i_{2r}k_{2r}}\\
 \qquad{}=\kappa_m^{\varphi_x} (t_{1},\dots, t_{1}^*)\delta_{i_1=i_{2r},i_2=i_3, \dots}1_{U_n^{nc}}
 =\kappa_m^{\varphi_x} (t_{i_1},\dots, t_{i_m}^*)1_{U_n^{nc}}.
\end{gather*}
Here, the first equation follows again from \Cref{cor::cum_of_cond}. Moreover, we used our assumptions on the cumulants $\kappa_m^{\varphi_x}$ for the second and fourth equation. In addition, the second equation uses the traciality of~$h_n$.

Similarly, if $m=2r$ and $\underline{e}=(*,\varnothing,\dots,*,\varnothing)$ is alternating, we have
\[\kappa_{m}^{E} (\alpha_n(t_{i_1}^*), \dots , \alpha_n(t_{i_m}))=\kappa_m^{\varphi_x} (t_{i_1}^*,\dots, t_{i_m})1_{U_n^{nc}}.\]
Finally, we have shown that the equality
\[\kappa_{m}^{E} \big(\alpha_n\big(t_{i_1}^{e_1}\big), \dots , \alpha_n\big(t_{i_m}^{e_m}\big)\big)=\kappa_m^{\varphi_x} \big(t_{i_1}^{e_1},\dots ,t_{i_m}^{e_m}\big)1_{U_n^{nc}}\]
is always true, which means that $E\circ \alpha_n=\varphi_x 1_{U_n^{nc}}$.

Implication $(2)\rightarrow (3)$:
If $E_{h_n}^{h_n \ast \varphi_x}\circ \alpha_n=\varphi_x 1_{U_n^{nc}}$, we have just to apply $h_n$ in order to get $(h_n \ast \varphi_x)\circ \alpha_n=\varphi_x$.
\end{proof}

In direct comparison with Curran's de Finetti theorem \cite{curran} for the unitary quantum group~$U_n^+$, we observe that our de Finetti theorem has a characterization of distributional invariance of \emph{finite} sequences~-- whereas in Curran's de~Finetti theorem, we only have a characterization of \emph{infinite} sequences.

\begin{Example} \label{ex-DeFin1}
Let us give several examples of sequence that satisfy the conditions of the de Finetti \Cref{thm-DeFin1}.
\begin{enumerate}\itemsep=0pt
\item As a finite sequence, we can take the elements of the first column of the matrix of generators of $\operatorname{Pol}(U_n^{\rm nc})$, i.e., $x_i=u_{i1}$, $i=1,\dots,n$, equipped with the Haar trace of $U_n^{\rm nc}$. It follows from \Cref{cor_CU15} that the distribution of $(x_i)_{i=1,\dots,n}$ satisfies the first condition of \Cref{thm-DeFin1} and consequently is invariant under the dual group action $\alpha_n$.
\item A sequence of free centered circular elements is invariant under the dual group action of $\operatorname{Pol}(U_n^{\rm nc})$. Conversely, if the $x_i$'s are invariant and free, then they are circular, because freeness implies vanishing of all cumulants $\kappa_{2r}(x_{i_1}^*, x_{i_1}, \dots, x_{i_r}^*, x_{i_r})$ and $\kappa_{2r}(x_{i_1}, x_{i_2}^*, x_{i_2},\dots,$ $x_{i_r}^*,x_{i_r}, x_{i_1}^*)$ unless $r=1$.
\item If the distribution of $(u_i)_{i=1,\dots,n}$ is invariant under the dual group action $\alpha_n$, and $x$ is $*$-free from $(u_i)_{i=1,\dots,n}$, then $(u_ix)_{i=1,\dots,n}$ is invariant under the dual group action $\alpha_n$. Indeed, \Cref{cum_alternating} says that the free cumulants
\[\kappa_n\big((u_{i(1)}x)^{e_1},\dots,(u_{i(n)}x)^{e_n}\big)\]
is vanishing if $(e_1,\dots,e_n)$ is not alternating, and, in the case where $(e_1,\dots,e_n)$ is alternating, \Cref{cum_of_prod} allows to says that the joint free cumulants are zero except those of type $\kappa_{2r} (x_{i_1}^*, x_{i_1}, \dots, x_{i_r}^*, x_{i_r} )$ or those of type $\kappa_{2r} (x_{i_1}, x_{i_2}^*, x_{i_2},\dots, x_{i_r}^*, x_{i_r}, x_{i_1}^* )$. Moreover, these cumulants depend only on the length $2r$.
\item If $(x_i)_{i=1,\dots,n}$ is $*$-free from $(u_{ij})_{i,j=1,\dots,n}\in U_n^{\rm nc}$, then the distribution of the tuple
\[
\Bigg(\sum_{j=1}^n u_{ij}x_j\Bigg)_{i=1,\dots,n}
\]
is invariant under the dual group action $\alpha_n$. Indeed, the distribution of $\big(\sum_{j=1}^n u_{ij}x_j\big)_{i=1,\dots,n}$ is given by $(h_n\ast \varphi_x)\circ \alpha_n$, and
\begin{align*}
 (h_n*((h_n*\varphi_x)\circ \alpha_n))\circ \alpha_n&=(h_n*h_n*\varphi_x)\circ (({\rm id}_{\operatorname{Pol}(U_n^{\rm nc})} \text{\b{$\sqcup$}} \alpha_n)\circ \alpha_n)\\
 &=(h_n*h_n*\varphi_x)\circ ((\Delta \text{\b{$\sqcup$}} {\rm id}_{\mathcal{Q}})\circ \alpha_n)
 =(h_n*\varphi_x)\circ \alpha_n,
\end{align*}
which means that the last condition of \Cref{thm-DeFin1} is satisfied.
\end{enumerate}
\end{Example}

\subsection{Invariance for finite sequences in the tracial case}\label{sectfinitetracial}

Let us now refine Theorem \ref{thm-DeFin1} in the case when our probability space $(A,\varphi)$ is tracial. We prepare the statement with the following proposition.

\begin{Proposition}\label{determinating_distribution}The $*$-distribution of a family $(x_1,\dots,x_n)$ of random variables in a noncommutative tracial probability space $(A,\varphi)$ which is invariant under the dual action $\alpha_n$ is uniquely determined by the distribution of $\sum_{i=1}^nx_i^*x_i$.
\end{Proposition}
\begin{proof}
Note first that by \Cref{thm-DeFin1}, the $*$-distribution of $(x_1,\dots,x_n)$ is uniquely determined by the sequences of cumulants $(\alpha_m)_{m\geq 1}$, and $(\beta_m)_{m\geq 1}$ where
\[\alpha_m:=\kappa_{2m}(x_{i_1}^*,x_{i_1},\dots,x_{i_m}^*,x_{i_m}),\qquad m\geq 1\]
and
\[\beta_m:=\kappa_{2m}(x_{i_1},\dots,x_{i_m}^*,x_{i_m},x_{i_1}^*),\qquad m\geq 1.\]
However, by traciality, we have that $(\alpha_m)_{m\geq 1}=(\beta_m)_{m\geq 1}$.

Thanks to \Cref{cum_of_prod}, let us compute,
\begin{align*}
 \kappa_{m}\left(\sum_{i=1}^nx_i^*x_i,\dots, \sum_{i=1}^nx_i^*x_i\right)&=\sum_{i_1,\dots,i_m=1}^n \kappa_{2m}(x_{i_1}^*x_{i_1},\dots, x_{i_m}^*x_{i_m})\\
 &=\sum_{i_1,\dots,i_m=1}^n\sum_{\substack{\pi\in \operatorname{NC}(2m)\\\pi \vee \sigma=1_{2m}}}\kappa_\pi(x_{i_1}^*,x_{i_1},\dots, x_{i_m}^*,x_{i_m}),
\end{align*}
 with $\sigma=\left\{\{1,2\}, \dots, \{2m-1, 2m\}\right\}$. Exactly as in the proof of \cite[Proposition~15.6]{nica+speicher}, all the cumulants appearing in the sum are of the form $\alpha_r$ (or $\beta_r$ which is in our tracial case the same as $\alpha_r$) for $r\leq m$. In fact, by taking out the term $\pi=1_{2m}$, we have more precisely
 \begin{gather*}
 \kappa_{m}\left(\sum_{i=1}^nx_i^*x_i,\dots, \sum_{i=1}^nx_i^*x_i\right)
=n^{m}\alpha_{m}
 +\sum_{i_1,\dots,i_m=1}^n \ \sum_{\substack{\pi\in \operatorname{NC}(2m)\setminus \{1_{2m}\}\\\pi \vee \sigma=1_{2m}}}\kappa_\pi(x_{i_1}^*,x_{i_1},\dots, x_{i_m}^*,x_{i_m}),
\end{gather*}
where all the cumulants appearing in last sum are of the form $\alpha_r$ for $r< m$. This formula can be inductively resolved for $(\alpha_m)_{m\geq 1}$ in terms of the cumulants of $\sum_{i=1}^nx_i^*x_i$, which shows that the $*$-distribution of $(x_1,\dots,x_n)$ is uniquely determined by the distribution of $\sum_{i=1}^nx_i^*x_i$.
\end{proof}

\begin{Proposition}Let us consider the $*$-algebra $\mathcal{S}_{n-1}^{\rm nc}$ defined by the quotient of $\mathcal{Q}_n$ by the relation $\sum_{i=1}^nt_i^*t_i=1$. There exists a unique tracial $*$-distribution on $\mathcal{S}_{n-1}^{\rm nc}$ which is invariant under the dual action $\alpha_n$.
\end{Proposition}

\begin{proof}
The uniqueness is due to the last proposition. The existence is due to the fact that the $*$-dis\-tribution of the first column $(u_{i1})_{1\leq i \leq n}$ is such an example of $*$-distribution, see Ex\-am\-ple~\ref{ex-DeFin1}.
\end{proof}

When the $*$-distribution of a family $(x_1,\dots,x_n)$ such that $\sum_{i=1}^nx_i^*x_i=1$ follows this particular $*$-distribution, we say that it is a \emph{freely uniform unit vector} of random variables. For example, the elements of one of the columns of the matrix of generators of $\operatorname{Pol}(U_n^{\rm nc})$, i.e., $x_i=u_{ik}$, $i=1,\dots,n$ and $k$ fixed, equipped with the Haar trace of $U_n^{\rm nc}$, is a freely uniform unit vector.

We will now prove two versions of a finite de Finetti theorem for tracial probability spaces $(A,\varphi)$.

\begin{Proposition}Let $(x_1,\dots,x_n)$ be random variables in a tracial probability space $(A,\varphi)$. Then the following statements are equivalent.
\begin{enumerate}\itemsep=0pt
 \item[$1.$] The $*$-distribution of $(x_1,\dots,x_n)$ is invariant under the dual action $\alpha_n$.
 \item[$2.$] The tuple $(x_1,\dots,x_n)$ has the same $*$-distribution as $(u_1x,\dots,u_nx)$ where $(u_1,\dots,u_n)$ is a freely uniform unit vector, $x$ is self-adjoint and $(u_1,\dots,u_n)$ and $x$ are $*$-free.
\end{enumerate}
In this case, $x^2$ and $\sum_{i=1}^nx_i^*x_i$ are identically distributed. More generally, the distribution of $x$ can be taken as any distribution such that $x^2$ and $\sum_{i=1}^nx_i^*x_i$ are identically distributed.\label{prop:realization}
\end{Proposition}

\begin{proof}Implication $(2)\rightarrow (1)$:
It is just an application of \Cref{ex-DeFin1}.

Implication $(1)\rightarrow (2)$:
By enlarging $(A,\varphi)$ if necessary, we consider a freely uniform unit vector $(u_1,\dots,u_n)$ and a~self-adjoint variable $x$, free from $(u_1,\dots,u_n)$, and such that $x^2$ and $\sum_{i=1}^nx_i^*x_i$ are identically distributed.
Thanks to \Cref{ex-DeFin1}, the family $(u_1x,\dots,u_nx)$ is invariant under the dual action $\alpha_n$ and the distribution of $\sum_{i=1}^n (u_ix)^*u_ix =x^2$ is the one of $\sum_{i=1}^nx_i^*x_i$. As a consequence of \Cref{determinating_distribution}, the families $(u_1x,\dots,u_nx)$ and $(x_1,\dots,x_n)$ have the same $*$-distribution.
\end{proof}

If the square root $\sqrt{\sum_{i=1}^nx_i^*x_i}$ exists, for example in a tracial $C^*$-probability space $(A,\varphi)$, the distribution of $x$ can be taken as the distribution of $\sqrt{\sum_{i=1}^nx_i^*x_i}$.

\begin{Corollary}\label{cor-finite}
Let $(x_1,\dots,x_n)$ be random variables in a tracial $W^*$-probability space $(A,\varphi)$ such that $\sum_{i=1}^nx_i^*x_i$ has a trivial kernel. Then the following statements are equivalent.
\begin{enumerate}\itemsep=0pt
 \item[$1.$] The $*$-distribution of $(x_1,\dots,x_n)$ is invariant under the dual action $\alpha_n$.
 \item[$2.$] We have the decomposition $(x_1,\dots,x_n)=(u_1x,\dots,u_nx)$ where $(u_1,\dots,u_n)$ is a freely uniform unit vector in $A$, which is $*$-free from $x:=\sqrt{\sum_{i=1}^nx_i^*x_i}$.
\end{enumerate}
\end{Corollary}

\begin{proof}
Implication $(2)\rightarrow (1)$: It is just an application of \Cref{ex-DeFin1}.

Implication $(1)\rightarrow (2)$:
 The fact that $x=\sqrt{\sum_{i=1}^nx_i^*x_i}$ has a trivial kernel implies that we can invert it (in the algebra of affiliated operators) and we can define $u_i:=x_i\cdot x^{-1}$ in such a way that $x_i=u_ix$. It remains to prove that $u_i$ and $x$ are $*$-free with the $*$-distribution announced.

 Let $(\tilde{u}_1\tilde{x},\dots,\tilde{u}_n\tilde{x})$ be the realization of the $*$-distribution of $(x_1,\dots,x_n)$ (not necessarily in~$A$) appearing in \Cref{prop:realization} with $\tilde{x}$ positive. But this means that the von Neumann algebra generated by $(x_1,\dots,x_n)$ is isomorphic to the von Neumann algebra generated by $(\tilde{u}_1\tilde{x},\dots,\tilde{u}_n\tilde{x})$ via the mapping $x_i\mapsto \tilde{u}_i\tilde{x}$. We extend this mapping to the algebra of affiliated operators (not necessarily bounded). The image of $x$ is $\sqrt{\sum_{i=1}^n(\tilde{u}_i\tilde{x})^*\tilde{u}_i\tilde{x}}=\tilde{x}$ and the image of $u_i=x_i\cdot x^{-1}$ is $\tilde{u}_i\tilde{x}\cdot x^{-1}$. As a consequence, the $*$-distribution of $(u_1,\dots,u_n, x)$ is the $*$-distribution of $(\tilde{u}_1,\dots,\tilde{u}_n, \tilde{x})$: they are bounded, and $(u_1,\dots,u_n)$ is a freely uniform unit vector in $A$, which is $*$-free from $x:=\sqrt{\sum_{i=1}^nx_i^*x_i}$.
\end{proof}

Note that the condition on the kernel $\sum_{i=1}^nx_i^*x_i$ can not be avoided if we want to define the unit vector $(u_1,\dots,u_n)$ in $A$. For example, the vector $(0,0,\dots,0)$ is invariant under the dual action $\alpha_n$ even if the tracial $W^*$-probability space $(A,\varphi)$ does not contain any freely uniform unit vector $(u_1,\dots,u_n)$.

\section{Infinite de Finetti theorems for dual group actions}

We now turn to the characterization of infinite sequences, building on our finite de Finetti Theorem~\ref{thm-DeFin1}. We will prove several variants: a general case, as a direct consequence of our finite de Finetti theorem; a version adapted to von Neumann algebras, i.e., to $W^*$-probability spaces; and, as in the previous section, a strengthening in the tracial case.

\subsection{Invariance for infinite sequences in the general case}
\begin{Definition}\label{def-inv-dg}
Let $(x_i)_{i \in \mathbb{N}}$ be a sequence of random variables in a noncommutative probability space $(A, \varphi)$. The distribution $\varphi_x$ of $(x_i)_{i \in \mathbb{N}}$ is said to be \emph{invariant under the dual action of $U^{\rm nc}$}, if, for any $n \ge 1$, $(x_1, \dots, x_n)$ is invariant under $\alpha_n$.
\end{Definition}

As a direct consequence of our finite de Finetti Theorem~\ref{thm-DeFin1}, we obtain the following characterization of infinite sequences under the action of the Brown algebra.

\begin{Theorem}\label{thm-infinitedual}
Let $(x_i)_{i \in \mathbb{N}}$ be a sequence of random variables in a noncommutative probability space $(A, \varphi)$. The following are equivalent:
\begin{enumerate}\itemsep=0pt
\item[$1.$]The distribution $\varphi_x$ of $(x_i)_{i \in \mathbb{N}}$ is invariant under the dual action of $U^{\rm nc}$. %\label{condInv}
\item[$2.$] The sequence $(x_i)_{i \in \mathbb{N}}$ is composed of $R$-diagonal elements such that the joint free cumulants are zero except those of type $\kappa_{2r}(x_{i_1}^*, x_{i_1}, \dots, x_{i_r}^*, x_{i_r})$ and $\kappa_{2r}(x_{i_1}, x_{i_2}^*, x_{i_2},\dots, x_{i_r}^*, x_{i_r}, x_{i_1}^*)$. Moreover, these cumulants depend only on the length $2r$. %\label{condCumul}
\end{enumerate}
\end{Theorem}

\begin{proof}The equivalence between the first two conditions is a direct consequence of \Cref{thm-DeFin1}.

If the first condition is true, for any $n \ge 1$, $(x_1, \dots, x_n)$ is invariant under $\alpha_n$. As a consequence, for any $n \ge 1$, the second condition holds for $(x_1, \dots, x_n)$, which implies that the second condition is true since $n$ can be as large as wanted.

Conversely, if the second condition is true, we have in particular that, for any $n \ge 1$, $(x_1, \dots, x_n)$ is invariant under $\alpha_n$.
\end{proof}

\subsection{A technical lemma on actions on infinitely many variables}

We extend the action $\alpha_n$ from Definition \ref{finitedualaction} to the infinite situation as follows.
Let $(x_i)_{i \in \mathbb{N}}$ be a sequence of random variables in a noncommutative probability space $(A, \varphi)$.
Define $\mathcal{Q}_\infty$ as the $*$-algebra of noncommutative polynomials with infinitely many variables $(t_j)_{j\geq 1}$ and complex coefficients.

Define $\beta_n \colon \mathcal{Q}_\infty \to \operatorname{Pol}(U^{\rm nc}_n) \sqcup \mathcal{Q}_\infty$ as the unital $*$-homomorphism satisfying
\[
\beta_{n}(t_i)=
\begin{cases}\displaystyle \sum_{j = 1}^n u_{ij} t_j & \text{if } 1 \leq i \leq n,\\
t_i & \text{if } i > n.
\end{cases}
\]
\begin{Lemma}\label{beta_invariance}If the distribution $\varphi_x$ of $(x_i)_{i \in \mathbb{N}}$ is invariant under the dual action of $U^{\rm nc}$, then, for all $n\geq 1,$
\[E_{h_n}^{h_n * \varphi_x}\circ \beta_n=\varphi_x.\]
\end{Lemma}

\begin{proof}We want to prove that for all $ m \geq 1$, $(i_1, \dots, i_m) \in \mathbb{N}^m$, $\underline{e} = (e_1, \dots, e_m) \in \{\varnothing, *\}^m$, we have
\[E_{h_n}^{h_n * \varphi_x}\big[\beta_n\big(t_{i_1}^{e_1}\big)\cdots \beta_n\big(t_{i_m}^{e_m}\big)\big]=\varphi_x\big[t_{i_1}^{e_1}\cdots t_{i_m}^{e_m}\big].\]
Set $N:=\max(n,i_1,\dots,i_m)$. We define $(v_{ij})_{1\leq i,j\leq N}$ by
\[
v_{ij} =
\begin{cases} u_{ij} & \text{if } 1 \leq i,j \leq n,\\
\delta_{ij} 1_{\operatorname{Pol}(U_n^{\rm nc})} & \text{if } \max\{i,j\} > n
\end{cases}
\]
in such a way that, for all $i\in [N],$ we have
\[\beta_n(t_i)=\sum_{k=1}^Nv_{ik}t_k.\]
We set $E:=E_{h_n}^{h_n \ast \varphi_x}$. Let $ m \geq 1$, $(i_1, \dots, i_m) \in [N]^m$, $\underline{e} = (e_1, \dots, e_m) \in \{\varnothing, *\}^m$.
We can use \Cref{cor::cum_of_cond} in order to compute
\begin{align*}
 \kappa_{m}^{E} \big(\beta_n\big(t_{i_1}^{e_1}\big), \dots , \beta_n\big(t_{i_m}^{e_m}\big)\big)&=\sum_{k_1, \dots, k_m \in [N]^m}\kappa_{m}^{E} \big(\big(v_{i_1k_1}t_{k_1}\big)^{e_1}, \dots , \big(v_{i_mk_m}t_{k_m}\big)^{e_m}\big)\\
 &=\sum_{k_1, \dots, k_m \in [N]^m} \kappa_m^{\varphi_x} \big(t_{k_1}^{e_1},\dots, t_{k_m}^{e_m}\big)\times (*).
\end{align*}
Using the particular form of the free cumulants $\kappa_m^{\varphi_x} \big(t_{k_1}^{e_1},\dots, t_{k_m}^{e_m}\big)$, given by \Cref{thm-DeFin1}, we see that all the terms of the sum are vanishing if $\underline{e}$ is not alternating.

Let us examine the case where $m=2r$ and $\underline{e}=(\varnothing,*,\dots,\varnothing,*)$ is alternating.
\begin{align*}
 &\kappa_{m}^{E} (\beta_n(t_{i_1}), \dots , \beta_n(t_{i_m}^*))\\
 & \quad{} =\sum_{k_1, \dots, k_{2r} \in [N]^{2r}} \kappa_m^{\varphi_x} (t_{k_1},\dots, t_{k_m}^*) h_n(v_{i_2k_2}^*v_{i_3k_3})\cdots h_n(v_{i_{2r-2}k_{2r-2}}^*v_{i_{2r-1}k_{2r-1}})v_{i_1k_1}v^*_{i_{2r}k_{2r}}\\
 & \quad{} =\kappa_m^{\varphi_x} (t_{1},\dots, t_{1}^*) \sum_{\substack{k_1=k_{2r},\\ k_2=k_3, \ldots \in [N]^{2r}}} h_n(v_{i_3k_3}v_{i_2k_2}^*)\cdots h_n(v_{i_{2r-1}k_{2r-1}}v_{i_{2r-2}k_{2r-2}}^*)v_{i_1k_1}v^*_{i_{2r}k_{2r}}\\
 & \quad{} =\kappa_m^{\varphi_x} (t_{1},\dots, t_{1}^*)\delta_{i_1=i_{2r},i_2=i_3, \dots}1_{U_n^{nc}}
 =\kappa_m^{\varphi_x} (t_{i_1},\dots, t_{i_m}^*)1_{U_n^{nc}}.
\end{align*}
Similarly, if $m=2r$ and $\underline{e}=(*,\varnothing,\dots,*,\varnothing)$ is alternating, we have
\[\kappa_{m}^{E} (\beta_n(t_{i_1}), \dots , \beta_n(t_{i_m}^*))=\kappa_m^{\varphi_x} (t_{i_1},\dots, t_{i_m}^*)1_{U_n^{nc}}.\]
Finally, we have shown that the equality
\[\kappa_{m}^{E} \big(\beta_n\big(t_{i_1}^{e_1}\big), \dots , \beta_n\big(t_{i_m}^{e_m}\big)\big)=\kappa_m^{\varphi_x} \big(t_{i_1}^{e_1},\dots ,t_{i_m}^{e_m}\big)1_{U_n^{nc}}\]
is always true, which means that $E\circ \beta_n=\varphi_x 1_{U_n^{nc}}.$
\end{proof}

\subsection{Invariance for infinite sequences for von Neumann algebras}
We now put more structure on our noncommutative probability space, passing to $W^*$-probability spaces, and we prove a de Finetti theorem in this situation.

Let $(x_i)_{i\in\mathbb{N}}$ be an infinite sequence of random variables in some $W^*$-probability space $(M,\varphi)$ with faithful state.
Set $\mathcal{B}_0:=W^*(x_1,x_2,\dots)$. More generally, we set \[\mathcal{B}_n:=W^*\Bigg(\sum_{j=1}^nx_j^*x_j,x_{n+1},x_{n+2},\dots\Bigg).\]
We have $\mathcal{B}_{n+1}\subset \mathcal{B}_{n}$, and we set
\[\mathcal B_\infty:=\bigcap_{n\geq 1} \mathcal{B}_n.\]
We define $(\tilde{\beta_n}(x_i))_{i\geq 1}$ elements of the $W^*$-probability space $({W^*}(U^{\rm nc}_n) * \mathcal{B}_0,h_n* \varphi)$ by
\[
\tilde{\beta_n}(x_i)=
\begin{cases}\displaystyle \sum_{j = 1}^n u_{ij} x_j & \text{if } 1 \leq i \leq n,\\
x_i & \text{if } i > n.
\end{cases}
\]
By \cite[Theorem~IX.4.2]{Takesaki}, there exists a unique $\varphi$-preserving conditional expectation
\[
E^{h_n * \varphi}_\varphi\colon \ {W^*}(U^{\rm nc}_n) * \mathcal{B}_0\to \mathcal{B}_0.
\]

\begin{Lemma}\label{cum_aux}
We set $E:=E_{\varphi}^{h_n \ast \varphi}$.
For all $n\geq 1,$ $m\geq 1$, $(i_1,\dots,i_m)\in \mathbb{N}^m $, $(e_1, \dots, e_m) \in \{\varnothing, *\}^m$ and $b_1\dots b_m\in \mathcal{B}_n$,
\[\kappa^{E}_m\big[\tilde{\beta_n}(x_{i_1})^{e_1}b_1, \dots , b_{m-1} \tilde{\beta_n}(x_{i_m})^{e_m}\big]\]
is in $\mathcal{B}_n$. In the case where $n\geq \max(i_1,\dots,i_m)$, we have more precisely the following:
\begin{itemize}\itemsep=0pt
 \item If $\underline{e}$ is not alternating, the cumulant is vanishing.
\item If $m$ is even and $\underline{e}=(\varnothing,*,\dots,\varnothing,*)$, we have
\begin{align*}
 &\kappa_{m}^{E} \big(\tilde{\beta_n}(x_{i_1})b_1, \dots , b_{m-1} \tilde{\beta_n}(x_{i_m})^*\big)\\
 &\qquad{}=(-1)^{m/2-1}n^{1-m/2}C_{m/2-1}\delta_{i_1=i_{m},i_2=i_3, \dots} \varphi(x_{1}b_1x_{1}^*) \varphi(b_2) \cdots \varphi(x_{1}b_{m-1}x_{1}^*).
\end{align*}
\item If $m$ is even and $\underline{e}=(*,\varnothing,\dots,*,\varnothing)$, we have
\begin{align*}&\kappa_{m}^{E} \big(\tilde{\beta_n}(x_{i_1})^*b_1, \dots , b_{m-1} \tilde{\beta_n}(x_{i_m})\big)\\
&\qquad{}=(-1)^{m/2-1}n^{-m/2}C_{m/2-1}\delta_{i_1=i_{2},i_3=i_4, \dots} \\
&\qquad\quad{} \times \varphi(b_1) \varphi(x_{1}b_2x_{1}^*)\cdots \varphi(x_{1}b_{m-2}x_{1}^*) \varphi(b_{m-1})\sum_{j=1}^n x_{j}^*x_{j} .
\end{align*}
\end{itemize}
\end{Lemma}
\begin{proof}Set $N:=\max(n,i_1,\dots,i_m)$. We define $(v_{ij})_{1\leq i,j\leq N}$ by
\[
v_{ij} =
\begin{cases} u_{ij} & \text{if } 1 \leq i,j \leq n,\\
\delta_{ij} 1_{\operatorname{Pol}(U_n^{\rm nc})} & \text{if } \max\{i,j\} > n
\end{cases}
\]
in such a way that, for all $i\in [N],$ we have
\[\tilde{\beta}_n(x_i)=\sum_{k=1}^Nv_{ik}x_k.\]
We can use \Cref{cor::cum_of_cond} in order to compute
\[\kappa_{m}^{E} \big(\tilde{\beta_n}(x_{i_1})^{e_1}b_1, \dots , b_{m-1} \tilde{\beta_n}(x_{i_m})^{e_m}\big)=\sum_{k_1, \dots, k_m \in [N]^m} \kappa_m^{h_n} \big(v_{i_1k_1}^{e_1},\dots, v_{i_mk_m}^{e_m}\big)\times (*),\]
where the term $(*)$ is dependent of $\underline{i}$, $\underline{e}, b_1, \dots,b_{m-1}$ and $k_1, \dots, k_m$. We know the free cumulants of $(u_{ij})_{1\leq i,j\leq n}$ (given by \Cref{cor_CU15}), and we deduce that the free cumulants of $(v_{ij})_{1\leq i,j\leq n}$ are vanishing except if $m=1$ or if $m$ is even with $\underline{e}$ alternating. It yields that all the terms of the sum are vanishing except if $m=1$ or if $m$ is even with $\underline{e}$ alternating.

Let us examine the case where $m=1$:
\[appa_{m}^{E} \big(\tilde{\beta_n}(x_{i_1})\big)=\sum_{k_1\in [N]}\kappa_m^{h_n}(v_{i_1k_1})x_{k_1}=\begin{cases} 0& \text{if } 1 \leq i_1 \leq n,\\ x_{i_1} & \text{ if } i_1 > n.\end{cases}\]
Let us examine the case where $m=2r$ and $\underline{e}=(\varnothing,*,\dots,\varnothing,*)$ is alternating,
\begin{align*}
 &\kappa_{m}^{E} \big(\tilde{\beta_n}(x_{i_1})b_1, \dots , b_{m-1} \tilde{\beta_n}(x_{i_m})^*\big)\\
 &\qquad{} =\sum_{k_1, \dots, k_{m} \in [N]^{m}} \kappa_m^{h_n} (v_{i_1k_1},\dots, v_{i_mk_m}^*) \varphi(x_{k_1}b_1x_{k_2}^*) \varphi(b_2)\cdots \varphi(x_{k_{m-1}}b_{m-1}x_{k_{m}}^*)1_M.
\end{align*}Finally, if $m=2r$ and $\underline{e}=(*,\varnothing,\dots,*,\varnothing)$ is alternating, we have
\begin{align*}
&\kappa_{m}^{E} \big(\tilde{\beta_n}(x_{i_1})^*b_1, \dots , b_{m-1} \tilde{\beta_n}(x_{i_m})\big)\\
&\qquad{} = \sum_{k_1, \dots, k_{m} \in [N]^{m}}\kappa_m^{h_n} (v_{i_1k_1}^*,\dots, v_{i_mk_m}) \varphi(b_1)\varphi(x_{k_2}b_2x_{k_3}^*)\cdots \varphi(b_{m-1})x_{k_1}^*x_{k_m},\end{align*}
which is vanishing if $i_1>n$ or $i_m>n$. When $0\leq i_1,i_m \leq n$, $k_1$ must equal $k_m$ and we have{\samepage
\begin{align*}
&\kappa_{m}^{E} \big(\tilde{\beta_n}(x_{i_1})^*b_1, \dots , b_{m-1} \tilde{\beta_n}(x_{i_m})\big) \\
&\qquad{} = \sum_{k=1}^nx_{k}^*x_{k} \sum_{k_2, \dots, k_{m-1} \in [N]^{m}}\kappa_m^{h_n} (v_{i_11}^*,v_{i_1k_2},\dots, v_{i_{m-1}k_{m-1}}^*,v_{i_m1}) \\
&\qquad\quad{} \times \varphi(b_1)\varphi(x_{k_2}b_2x_{k_3}^*)\cdots \varphi(b_{m-1}).
\end{align*}
In all cases, the cumulant $\kappa_{m}^{E} \big(\tilde{\beta_n}(x_{i_1})^{e_1}b_1, \dots , b_{m-1} \tilde{\beta_n}(x_{i_m})^{e_m}\big)$ belongs to~$\mathcal{B}_n$.}

Whenever $n\geq \max(i_1,\dots,i_m)$ (or equivalently $N=n$), we can pursue the computation.
In the case where $m=2r$ and $\underline{e}=(\varnothing,*,\dots,\varnothing,*)$ is alternating, we have
\begin{align*}
 &\kappa_{m}^{E} \big(\tilde{\beta_n}(x_{i_1})b_1, \dots , b_{m-1} \tilde{\beta_n}(x_{i_m})^*\big)\\
 &\qquad{} =\sum_{k_1, \dots, k_{m} \in [n]^{m}} \kappa_m^{h_n} (u_{i_1k_1},\dots, u_{i_mk_m}^*) \varphi(x_{k_1}b_1x_{k_2}^*) \varphi(b_2)\cdots \varphi(x_{k_{m-1}}b_{m-1}x_{k_{m}}^*)\\
 &\qquad{} =(-1)^{m/2-1}n^{1-m}C_{m/2-1}\delta_{i_1=i_{m},i_2=i_3, \dots}\\
 &\qquad\quad{} \times \sum_{k_1=k_{2},k_3=k_4, \ldots \in [n]^{m}} \varphi(x_{k_1}b_1x_{k_2}^*)\varphi(b_2) \cdots\varphi(x_{k_{m-1}}b_{m-1}x_{k_{m}}^*)\\
 &\qquad{} =(-1)^{m/2-1}n^{1-m/2}C_{m/2-1}\delta_{i_1=i_{m},i_2=i_3, \dots} \varphi(x_{1}b_1x_{1}^*) \varphi(b_2) \cdots \varphi(x_{1}b_{m-1}x_{1}^*).
\end{align*}
Similarly, if $m=2r$ and $\underline{e}=(*,\varnothing,\dots,*,\varnothing)$ is alternating, we have
\begin{align*}
&\kappa_{m}^{E} \big(\tilde{\beta_n}(x_{i_1})^*b_1, \dots , b_{m-1} \tilde{\beta_n}(x_{i_m})\big)\\
&\qquad{} =(-1)^{m/2-1}n^{1-m}C_{m/2-1}\delta_{i_1=i_2,i_3=i_4, \dots} \\
 &\qquad\quad{} \times \sum_{k_1=k_{m},k_2=k_3, \ldots \in [n]^{2r}} x_{k_1}^*x_{k_{m}} \varphi(b_1) \varphi(x_{k_2}b_2x_{k_3}^*)\cdots
\varphi(b_{m-1})\\
&\qquad{} =(-1)^{m/2-1}n^{-m/2}C_{m/2-1}\delta_{i_1=i_2,i_3=i_4, \dots} \\
 &\qquad\quad{} \times \varphi(b_1) \varphi(x_{1}b_2x_{1}^*)\cdots \varphi(x_{1}b_{m-2}x_{1}^*) \varphi(b_{m-1})\sum_{j=1}^n x_{j}^*x_{j}.\tag*{\qed}
 \end{align*}\renewcommand{\qed}{}
\end{proof}

During the rest of this section, we assume that the distribution $\varphi_x$ of $(x_i)_{i \in \mathbb{N}}$ is invariant under the dual action of $U^{\rm nc}$. Thanks to \Cref{beta_invariance}, the $*$-distribution of $\big(\tilde{\beta_n}(x_j)\big)
_{j\geq 1}$ and the $*$-distribution of $(x_j)
_{j\geq 1}$ are the same, which means that we can extend $\tilde{\beta_n}$ to a homomorphism from $\mathcal{B}_0$ to ${W^*}(U^{\rm nc}_n) * \mathcal{B}_0$ such that $(h_n * \varphi)\circ \tilde{\beta_n}=\varphi$.
\begin{Lemma}\label{beta_inv}The linear map $E_n:=E^{h_n * \varphi}_\varphi\circ \tilde{\beta_n}$ is a $\varphi$-preserving conditional expectation from $\mathcal{B}_0$ to $\mathcal{B}_n$.
\end{Lemma}
\begin{proof}
We have $\varphi\circ E^{h_n * \varphi}_\varphi\circ \tilde{\beta_n}=h_n * \varphi\circ \tilde{\beta_n}=\varphi$. Moreover, because
\[\sum_{j=1}^nx_j^*x_j =\sum_{j=1}^n\tilde{\beta_n}(x_j)^*\tilde{\beta_n}(x_j),\qquad x_{n+1} =\tilde{\beta_n}(x_{n+1}),\qquad x_{n+2}=\tilde{\beta_n}(x_{n+2}),\qquad \dots,\]
we know that $\tilde{\beta_n}$ is the identity on $\mathcal{B}_n$, and we can write the bimodule property: for all $a\in \mathcal{B}_0$, and $b_1,b_2\in \mathcal{B}_n$,
\[E^{h_n * \varphi}_\varphi\circ \tilde{\beta_n}[b_1ab_2]=E^{h_n * \varphi}_\varphi\big[b_1\tilde{\beta_n}(a)b_2\big]=b_1E^{h_n * \varphi}_\varphi\big[\tilde{\beta_n}(a)\big]b_2.\]
It remains to prove that $E_n$ takes value in $\mathcal{B}_n$, which is true because
\[E^{h_n * \varphi}_\varphi\circ \tilde{\beta}_n\big[(x_{i_1})^{e_1}\cdots (x_{i_m})^{e_m}\big]=E^{h_n * \varphi}_\varphi\big[\tilde{\beta_n}(x_{i_1})^{e_1}\cdots \tilde{\beta_n}(x_{i_m})^{e_m}\big]\in \mathcal{B}_n ,\]
thanks to \Cref{cum_aux}.
\end{proof}
\begin{Proposition}[{\cite[Proposition~4.7]{curran}}]\label{conv_curran}For any $x\in M$, the sequence $E_n[x]$ converges in strong topology to a conditional expectation $E\colon\mathcal{B}_0\to \mathcal{B}_\infty$. Moreover, for all $m\geq 1$, $a_1,\dots,a_m\in M$, we have
\[\lim_{n\to \infty}\kappa_m^{E_n}(a_1, \dots , a_m)=\kappa_m^{E}(a_1, \dots , a_m).\]
\end{Proposition}

Here comes our infinite de Finetti theorem in the case of von Neumann algebras, i.e., for $W^*$-probability spaces.

\begin{Theorem}\label{theorem:infinite_setting}
Let $(x_i)_{i\in\mathbb{N}}$ be an infinite sequence of random variables in some $W^*$-probability space $(M,\varphi)$. The following are equivalent:
\begin{enumerate}\itemsep=0pt
\item[$1.$] There exists $v\in M$ such that, setting $\mathbb{E}\colon M\to \mathcal{B}$ the conditional expectation from $M$ to $\mathcal{B}:=W^*(v)$, $(x_i)_{i\in\mathbb{N}}$ is a $\mathcal{B}$-valued free centered circular family whose elements have identical variances
\[
\mathcal{B}\ni b\mapsto \mathbb{E}(x_i b x_i^*)=\varphi(x_ibx_i^*)1_M \qquad\text{and}\qquad \mathcal{B}\ni b\mapsto \mathbb{E}(x_i^*bx_i)=\varphi(b)v.
\]
\item[$2.$] The distribution $\varphi_x$ of $(x_i)_{i \in \mathbb{N}}$ is invariant under the dual action of~$U^{\rm nc}$.
\end{enumerate}
In this case, the sequence $\big(\frac{1}{n}\sum_{j=1}^nx_j^*x_j\big)_{n\in \mathbb{N}}$ strongly converges to~$v$.
\end{Theorem}

\begin{proof}
Implication $(1)\rightarrow (2)$: First of all, the variables $x_j^*x_j$ are freely independent and identically distributed with respect to $\mathbb{E}$. As a consequence, the normalized sum $\big(\frac{1}{n}\sum_{j=1}^nx_j^*x_j\big)_n$ strongly converges to $\mathbb{E}[x_i^*x_i]=v$ thanks to the free law of large number.

Now the strategy is the following: we will consider a sequence of variables which are invariant under the dual action of $U^{\rm nc}$, and which converges in distribution to $(x_i)_{i \in \mathbb{N}}$.

More precisely, let us consider $\big(\tilde{\beta}_n(x_i)\big)_{i \in \mathbb{N}}$. We set $E:=E^{h_n * \varphi}_\varphi$. For all $n\geq 1,$ $m \geq 1$, $(i_1,\dots,i_m)\in \mathbb{N}^m $ and $b_1\dots b_m\in \mathcal{B}_\infty$, \Cref{cum_aux} gives us the exact value of
\[\kappa^{E}_m\big[\tilde{\beta}_n(x_{i_1})^{e_1}b_1, \dots , b_{m-1} \tilde{\beta}_n(x_{i_m})^{e_m}\big]\]
in the case where $n\geq \max(i_1,\dots,i_m)$. By letting $n$ tend to $\infty$, we get $0$ if $m\neq 2$, or if $e_1=e_2$. The only non-vanishing cases are
\[
\lim_{n\to \infty} \kappa^{E}_2\big[\tilde{\beta}_n(x_{i_1})b,\tilde{\beta}_n(x_{i_2})^*\big]= \delta_{i_1,i_2}\varphi(x_1bx_1^*)=\kappa^{\mathbb{E}}_2[x_{i_1}b,x_{i_2}^*]
\]
and
\[\lim_{n\to \infty} \kappa^{E}_2\big[\tilde{\beta}_n(x_{i_1})^*b,\tilde{\beta}_n(x_{i_2})\big]= \delta_{i_1,i_2}\varphi(b)v=\kappa^{\mathbb{E}}_2[x_{i_1}^*b,x_{i_2}].
\]
As a consequence, we can say that the free cumulants $\kappa^{E}_m$ of $\big(\tilde{\beta}_n(x_i)\big)_{i \in \mathbb{N}}$ under $E=E^{h_n * \varphi}_\varphi$ converge strongly to the free cumulant $\kappa^{\mathbb{E}}_m$ of $(x_i)_{i \in \mathbb{N}}$ under $\mathbb{E}$. Moreover, by induction, it is also true for the free cumulant $\kappa^{E}_{\pi}$ of a noncrossing partition $\pi$ which converges to the corresponding free cumulant $\kappa^{\mathbb{E}}_\pi$. It implies firstly that the distribution of $\big(\tilde{\beta}_n(x_i)\big)_{i \in \mathbb{N}}$ under $E^{h_n * \varphi}_\varphi$ converges strongly to the distribution of $(x_i)_{i \in \mathbb{N}}$ under $\mathbb{E}$ and secondly that the distribution of $\big(\tilde{\beta}_n(x_i)\big)_{i \in \mathbb{N}}$ under $h_n * \varphi$ converges strongly to the distribution of $(x_i)_{i \in \mathbb{N}}$ under~$\varphi$.

In order to conclude, we remark that, for any $1\leq m \leq n$, the distribution o
\[
\big(\tilde{\beta}_n(x_1),\dots,\tilde{\beta}_n(x_m)\big)
\]
is invariant under $\alpha_m$. As a consequence, for any $1\leq m$, the distribution of $(x_1,\dots,x_m)$ is invariant under $\alpha_m$ (as it is the limit of the distribution of $\big(\tilde{\beta}_n(x_1),\dots,\tilde{\beta}_n(x_m)\big)$ when $n$ tends to $\infty$).

Implication $(2)\rightarrow (1)$:
\Cref{beta_inv} tells us that
$E_n[x_1^*x_1]=\frac{1}{n}\sum_{j=1}x_j^*x_j$, and this variable converges to $v:=E[x_1^*x_1]$ thanks to \Cref{conv_curran}.

For all $n\geq 1$, $m\geq 1$, $(i_1,\dots,i_m)\in \mathbb{N}^m $ and $b_1\dots b_m\in \mathcal{B}_\infty$, \Cref{cum_aux} gives us the exact value of
\[\kappa^{E_n}_m\big[x_{i_1}^{e_1}b_1,\dots , b_{m-1} x_{i_m}^{e_m}\big]\]
in the case where $n\geq \max(i_1,\dots,i_m)$. By letting $n$ tend to $\infty$, we get $0$ if $m\neq 2$, or if $e_1=e_2$. The only non-vanishing cases are
\[
\lim_{n\to \infty}\kappa^{E_n}_2[x_{i_1}b,x_{i_2}^*]=\delta_{i_1,i_2}\varphi(x_1bx_1^*) \qquad\text{and}\qquad \lim_{n\to \infty}\kappa^{E_n}_2[x_{i_1}^*b,x_{i_2}]=\delta_{i_1,i_2}\varphi(b)v.
\]

\Cref{conv_curran} allows us to conclude that the cumulants $\kappa^{E}$ of $(x_i)$ are always vanishing, except
\[
\mathcal{B}_\infty\ni b\mapsto \kappa^{E}_2[x_{i}b,x_{i}^*]=\varphi(x_ibx_i^*)1_\mathcal{B} \qquad\text{and}\qquad \mathcal{B}_\infty\ni b\mapsto \kappa^{E}_2[x_{i}^*b,x_{i}]=\varphi(b)v,
\]
which means that $(x_i)_{i\in\mathbb{N}}$ is a $\mathcal{B}_\infty$-valued free centered circular family whose elements have identical variances. Because $\mathcal{B}$ is invariant by the action of these variances, we get the result.
\end{proof}

\subsection{Invariance for infinite sequences for tracial von Neumann algebras}

Just like in our finite de Finetti theorem (see Section~\ref{sectfinitetracial}), we now consider the tracial case, i.e., of $W^*$-probability spaces where $\varphi$ is a trace.

\begin{Proposition}\label{prop:realization_infinite}Let $(x_i)_{i\in\mathbb{N}}$ be an infinite sequence of random variables in some tracial $W^*$-probability space $(M,\varphi)$. The following are equivalent:
\begin{enumerate}\itemsep=0pt
 \item[$1.$] The distribution $\varphi_x$ of $(x_i)_{i \in \mathbb{N}}$ is invariant under the dual action of $U^{\rm nc}$.
 \item[$2.$] The sequence $(x_i)_{i \in \mathbb{N}}$ has the same $*$-distribution as $(c_ix)_{i \in \mathbb{N}}$ where $(c_i)_{i \in \mathbb{N}}$ is a sequence of free circular variables, $x$ is self-adjoint and $(c_i)_{i \in \mathbb{N}}$ and $x$ are $*$-free.
\end{enumerate}
In this case, $x^2$ and the strong limit $v$ of $\left(\frac{1}{n}\sum_{i=1}^nx_i^*x_i\right)_{n\in \mathbb{N}}$ are identically distributed. More generally, the distribution of $x$ can be taken as any distribution such that $x^2$ and $v$ are identically distributed.
\end{Proposition}
For example, the distribution of $x$ can be taken as the distribution of $\sqrt{v}$.
\begin{proof}Implication $(2)\rightarrow (1)$:
It is just an application of \Cref{ex-DeFin1}.

Implication $(1)\rightarrow (2)$:
By enlarging $(M,\varphi)$ if necessary, we consider a self-adjoint variable $x$ and a sequence of free circular variables $(c_i)_{i \in \mathbb{N}}$ $*$-free from $x$ such that $x^2$ and $v$ are identically distributed. Denoting by $E$ the conditional expectation from $M$ to $W^*(x)$, we can compute the $W^*(x)$-valued cumulants of $(c_ix,(c_ix)^*)_{i \in \mathbb{N}}$ thanks to Theorem~\ref{cum_of_cond}. They all vanish except
\[ \kappa_2^E(c_ixb,(xc_i)^*)=\varphi\big(x^2b\big)1_M \qquad\text{and}\qquad \kappa_2^E((c_ix)^*b,xc_i)=\varphi(b)x^2,\qquad \forall b\in W^*(x),
\] which means that $(c_ix)_{i \in \mathbb{N}}$ is a $W^*(x)$-valued free centered circular family. The variance leaving invariant the subalgebra $W^*\big(x^2\big)$, $(c_ix)_{i \in \mathbb{N}}$ is a $W^*\big(x^2\big)$-valued free centered circular family with variances
 \[
W^*\big(x^2\big)\ni b\mapsto \varphi\big(x^2b\big)1_M \qquad\text{and}\qquad W^*\big(x^2\big)\ni b\mapsto \varphi(b)x^2.
\]
Using Theorem~\ref{theorem:infinite_setting}, we know that $(x_i)_{i \in \mathbb{N}}$ is a $W^*(v)$-valued free centered circular family with variances
\[
W^*(v)\ni b\mapsto \varphi(x_ibx_i^*)1_M=\varphi(vb)1_M \qquad\text{and}\qquad W^*(v)\ni b\mapsto \varphi(b)v,
\]
where we used the traciality and the exchangeability to write
\[\varphi(x_ibx_i^*)1_M=\varphi(bx_i^*x_i)1_M=\lim_n\varphi\left(b\cdot \left(\frac{1}{n}\sum_{i=1}^{n}x_i^*x_i\right)\right)1_M=\varphi(vb)1_M .\]
The distribution of $v$ and $x^2$ being the same, $(x_i)_{i \in \mathbb{N}}$ and $(c_ix)_{i \in \mathbb{N}}$ have the same $*$-distri\-bu\-tion.
\end{proof}

 \begin{Corollary}\label{cor-infinite}
Let $(x_i)_{i\in\mathbb{N}}$ be an infinite sequence of random variables in some tracial $W^*$-probability space $(M,\varphi)$. The following are equivalent:
\begin{enumerate}\itemsep=0pt
 \item[$1.$] The distribution $\varphi_x$ of $(x_i)_{i \in \mathbb{N}}$ is invariant under the dual action of $U^{\rm nc}$ and the strong limit $v$ of $\big(\frac{1}{n}\sum_{i=1}^nx_i^*x_i\big)_{n\in \mathbb{N}}$ has a trivial kernel.
 \item[$2.$] We have the decomposition $(x_i)_{i \in \mathbb{N}}=(c_ix)_{i \in \mathbb{N}}$ where $(c_i)_{i \in \mathbb{N}}$ is a sequence of free circular variables in $M$, $x$ is the strong limit of $\big(\sqrt{\frac{1}{n}\sum_{i=1}^nx_i^*x_i}\big)_{n \in \mathbb{N}}$, $x$ has a trivial kernel and $(c_i)_{i \in \mathbb{N}}$ and $x$ are $*$-free.
\end{enumerate}
\end{Corollary}
\begin{proof}
Implication $(2)\rightarrow (1)$: It is just an application of \Cref{ex-DeFin1}, with the additional observation that $v=x^2$ and $x$ have the same kernel.

Implication $(1)\rightarrow (2)$: The fact that $v$ has a trivial kernel implies that we can invert $x:=\sqrt{v}$ which is the strong limit of $\big(\sqrt{\frac{1}{n}\sum_{i=1}^nx_i^*x_i}\big)_{n \in \mathbb{N}}$ (in the algebra of affiliated operators) and we can define $c_i:=x_i\cdot x^{-1}$ in such a way that $x_i=c_ix$. It remains to prove that $(c_i)_{i\in \mathbb{N}}$ and $x$ are in $M$ and are $*$-free with the $*$-distribution announced.

 Let $(\tilde{c}_i\tilde{x})_{i\in \mathbb{N}}$ be the realization of the $*$-distribution of $(x_i)_{i\in \mathbb{N}}$ (not necessarily in $M$) appearing in \Cref{prop:realization_infinite}, with $\tilde{x}$ positive. But this means that the von Neumann algebra generated by $(x_i)_{i\in \mathbb{N}}$ is isomorphic to the von Neumann algebra generated by $(\tilde{c}_i\tilde{x})_{i\in \mathbb{N}}$ via the mapping \mbox{$x_i \mapsto \tilde{c}_i\tilde{x}$}. We extend this mapping to the algebra of affiliated operators (not necessarily bounded). The image of $x$ is the strong limit of $\sqrt{\frac{1}{n}\sum_{i=1}^n(\tilde{c}_i\tilde{x})^*\tilde{c}_i\tilde{x}}$, which is $\tilde{x}$, and the image of $c_i=x_i\cdot x^{-1}$ is $\tilde{c}_i\tilde{x}\cdot \tilde{x}^{-1}=\tilde{c}_i$. As a consequence, the $*$-distribution of $(c_i, x)_i$ is the $*$-distribution of $(\tilde{c}_i, \tilde{x})$: they are bounded, and $(c_i)_{i\in \mathbb{N}}$ is a sequence of free circular variables in $M$, which is $*$-free from $x:=\sqrt{\sum_{i=1}^nx_i^*x_i}$.
\end{proof}

\section{De Finetti theorems for bialgebra actions}\label{sec-bialg}

We now pass to a different kind of action of the Brown algebra: to bialgebra actions. These are actions which make use of the tensor product of algebras (bialgebra actions) rather than of the free product (dual group actions). Surprisingly, there cannot be a de Finetti theorem in that case: we will show a kind of no-go theorem for this situation. However, if we weaken the assumption of a $W^*$-probability space to a space $(M,\varphi)$ where $\varphi$ is not faithful, we do obtain a~non-trivial de Finetti theorem.

\subsection{Bialgebra actions}

Denote again by $\mathcal{Q}_n$ the unital free $*$-algebra generated by $t_i$, $i=1,\dots,n$. The Brown algebra $\operatorname{Pol}(U_n^{\rm nc})$ has also an action as $*$-bialgebra on $\mathcal{Q}_n$, since for every $n\in\mathbb{N}$ there exists a unique $*$-homomorphism $\gamma_n \colon \mathcal{Q}_n\to\operatorname{Pol}(U_n^{\rm nc})\otimes\mathcal{Q}_n$ with $\gamma_n(t_i) = \sum_{j=1}^n u_{ij} \otimes t_j$, which furthermore satisfies the coaction identities
\[
(\Delta\otimes {\rm id})\circ\gamma_n = ({\rm id}\otimes \gamma_n)\circ\gamma_n
\text{ and } (\delta\otimes {\rm id})\circ\gamma_n = {\rm id}.
\]

\begin{Definition}\label{def-inv-bialg}
Let $(x_i)_{i\in\mathbb{N}}$ be a sequence of random variables in a noncommutative probability space $(A,\varphi)$. The distribution $\varphi_x$ of $(x_i)_{i\in\mathbb{N}}$ is said to be \emph{invariant under the $*$-bialgebraic action of $U^{\rm nc}$}, if $\varphi_x$ is invariant under the coactions $\gamma_n$, i.e., if
\[
({\rm id}\otimes\varphi_x)\circ \gamma_n = \varphi_x \mathbf{1},
\]
for all $n\ge 1$.
\end{Definition}

\begin{Remark}
A sequence $(x_i)_{i\in\mathbb{N}}$ of quantum random variables is invariant under the $*$-bialgebraic action of $U^{\rm nc}$ if and only if we have
\begin{equation}\label{eq-inv-bialg}
\sum_{1\le i_1,\dots,i_k\le n} u^{e_1}_{j_1i_1}\cdots u^{e_k}_{j_ki_k}\varphi\big(x^{e_1}_{i_1}\cdots x^{e_k}_{i_k}\big) = \varphi\big(x^{e_1}_{j_1}\cdots x^{e_k}_{j_k}\big)1,
\end{equation}
for all $k\in\mathbb{N}$, $1\le j_1,\dots,j_k\le n$, $e=(e_1,\dots,e_k)\in\{\varnothing,*\}^k$.
\end{Remark}

\subsection{No-go de Finetti theorem for faithful states}

In the case of usual $W^*$-probability spaces $(M,\varphi)$, where $\varphi$ is a faithful state, we prove that there exist no non-trivial sequences that are invariant under the $*$-bialgebraic action of $U^{\rm nc}$. This constitutes our no-go de Finetti theorem for the Brown algebra under these kind of actions.

\begin{Theorem}\label{thm-bialg-DeFin}
Let $(x_i)_{i\in\mathbb{N}}$ be an infinite sequence of random variables in some $W^*$-probability space $(M,\varphi)$.
The joint $*$-distribution $\varphi_x$ of $(x_i)_{i\in\mathbb{N}}$ is invariant under the $*$-bialgebraic action of $U^{\rm nc}$ if and only if $x_i=0$ for all $i\in\mathbb{N}$.
\end{Theorem}

In the proof of this theorem, we will use the following finite-dimensional representations of $\operatorname{Pol}(U^{\rm nc}_n)$.

\begin{Lemma}\label{lem-fd-rep}
There exists a unique unital $*$-homomorphism $\pi_n\colon \operatorname{Pol}(U^{\rm nc}_n)\to M_n(\mathbb{C})$ such that
\[
\pi_n(u_{jk}) = e_{kj}
\]
for $1\le j$, $k\le n$. This homomorphism does not factorize via the quotient map $q_n\colon \operatorname{Pol}(U^{\rm nc}_n)\to \operatorname{Pol}(U^+_n)$ for $n\ge 2$.
\end{Lemma}

\begin{proof}
Indeed, the assignment $u_{jk}\mapsto e_{kj}$, $u^*_{jk}\mapsto e_{jk}$ satisfies the two defining relations $uu^*=1=u^*u$ of $\operatorname{Pol}(U^{\rm nc}_n)$:
\begin{gather*}
\sum_k \pi_n(u_{ik})\pi_n(u^*_{jk}) = \delta_{ij} \sum_k e_{kk} = \delta_{ij} 1, \qquad
\sum_k \pi_n(u^*_{ki})\pi_n(u_{kj}) = \delta_{ij} \sum_k e_{kk} = \delta_{ij} 1.
\end{gather*}
But for $n\ge 2$ it does not satisfy the other two relations $u^t \overline{u}=1=\overline{u}u^t$ that define $\operatorname{Pol}(U^+_n)$, instead we have
\begin{gather*}
\sum_k \pi_n(u^*_{ik})\pi_n(u_{jk}) = n e_{ij}, \qquad
\sum_k \pi_n(u_{ki})\pi_n(u^*_{ki}) = n e_{ij}.\tag*{\qed}
\end{gather*}\renewcommand{\qed}{}
\end{proof}

\begin{proof}[Proof of \Cref{thm-bialg-DeFin}]
Let $n\ge 2$. Applying the $*$-representation $\pi_n$ defined in \Cref{lem-fd-rep} to invariance condition \eqref{eq-inv-bialg} for products of the form $x_j^*x_k$, $1\le j,k\le n$, we get
\[
\varphi(x_j^* x_k)1 = \gamma\Bigg(\sum_{i_1,i_2 = 1}^n u_{ji_1}^* u_{ki_2} \varphi(x_{i_1}^* x_{i_2})\Bigg) = e_{jk} \sum_{i=1}^n \varphi(x_i^* x_i),
\]
which implies in particular $\varphi(x_i^*x_i)=0$ and therefore, by faithfulness of $\varphi$, $x_i=0$ for all $1\le i \le n$.
\end{proof}

\subsection{Half a de Finetti theorem for non-faithful states}

Our proof of \Cref{thm-bialg-DeFin} depends in a crucial way on the assumption that the von Neumann algebra $M$ is equipped with a \emph{faithful} state. We will now show that there do exist sequences that are invariant under the $*$-bialgebraic action of $U^{\rm nc}$, if we weaken this assumption. Let us call a~pair $(M,\psi)$ of a von Neumann algebra equipped with a not necessarily faithful state $\psi$ a~\emph{weak} $W^*$-probability space. It is straightforward to extend the notions of a joint $*$-distribution in \Cref{def-jointdistribution} and of invariance under the $*$-bialgebraic action of $U^{\rm nc}$ in \Cref{def-inv-bialg} to weak $W^*$-probability spaces.

We have half a de Finetti theorem for the case of weak $W^*$-probability spaces, i.e., we can only prove one direction of the usual de Finetti theorems.

\begin{Proposition}\label{prop-bialg-DeFin}
Let $(x_i)_{i\in\mathbb{N}}$ be an infinite sequence of elements of a weak $W^*$-probability space $(M,\psi)$.

If there exists a $W^*$-subalgebra $\mathbf{1}\in B\subseteq M$ and a conditional expectation $E\colon M\to B$ such that $(x_i)_{i\in\mathbb{N}}$ is a $B$-valued free centered circular family whose elements have identical variances
\[
B\ni b\mapsto \theta(b)=E(x_ibx_i^*)\in B \qquad\text{and}\qquad B\ni b\mapsto \eta(b)=E(x_i^*bx_i)=0,
\]
for all $i\in\mathbb{N}$, then the joint distribution of $(x_i)_{i\in\mathbb{N}}$ is invariant under the $*$-bialgebraic action of~$U^{\rm nc}$.
\end{Proposition}

For $e=(e_1,\dots,e_{2k})\in\{\varnothing,*\}^{2k}$, denote by
\[
\operatorname{NC}^{e,(\varnothing,*)}_2(k) = \{\pi\in \operatorname{NC}_2(k);\forall \{s,t\}\in\pi, e_{\min(s,t)}=\varnothing,e_{\max(s,t)}=*\},
\]
the pair partitions whose pairs `join' a $\varnothing$ to a $*$.

We will need the following lemma.
\begin{Lemma}\label{lem-rel-Brown}
Let $k\ge 1$, $j=(j_1,\dots,j_{2k})\in\{1,\dots,n\}^{2k}$, $e=(e_1,\dots,e_{2k})\in\{\varnothing,*\}^{2k}$, and $\pi\in \operatorname{NC}^{e,(\varnothing,*)}_2(k)$. Then we have
\[
\sum_{1\le i_1,\dots,i_{2k}, \, \pi\preceq \operatorname{ker} i} u^{e_1}_{j_1i_1}\cdots u^{e_{2k}}_{j_{2k}i_{2k}} =
 \begin{cases}
1 & \text{if } \pi \preceq \operatorname{ker} j, \\
0 & \text{else}.
\end{cases}
\]
\end{Lemma}

\begin{proof}
We will prove this by induction.
For $k=1$ this is simply one of the defining relations of $\operatorname{Pol}(U_n^{\rm nc})$.

Suppose now $k>1$. Then $\pi$ contains an interval $V=(\ell,\ell+1)$ with $e_\ell=\varnothing$, $e_{\ell+1}=*$. Therefore
\begin{gather*}
\sum_{1\le i_1,\dots,i_{2k}, \,\pi\preceq \operatorname{ker} i} u^{e_1}_{j_1i_1}\cdots u^{e_{2k}}_{j_{2k}i_{2k}} \\
\qquad{} = \sum_{1\le i_1,\dots, i_{\ell-1},i_{\ell+2},\dots, i_{2k},\, \pi\backslash V\preceq \operatorname{ker} i} u^{e_1}_{j_1i_1}\cdots \underbrace{\left(\sum_{i=1}^n u_{j_\ell i} u^*_{j_{\ell+1} i}\right)}_{=\delta_{j_\ell j_{\ell+1}}}\cdots u^{e_{2k}}_{j_{2k}i_{2k}},
\end{gather*}
from which the result follows with the induction hypothesis.
\end{proof}

We will also need an expression for the conditional expectation of circular elements with one vanishing variance $\eta(b)= \kappa_{2}^E(x_i^*b, x_1)\equiv0$. From the last formula in Section~\ref{sect-opval}, we get in this case
\begin{equation}\label{eq-cum}
E\big(b_0x_{i_1}^{e_1}b_1\cdots x_{i_k}^{e_k}\big) =
 \begin{cases}
\displaystyle \sum_{\pi\in \operatorname{NC}^{e,(\varnothing,*)}_2(k),\, \pi\preceq \operatorname{ker} i} \kappa_{\pi}^E \big[b_0x_1b_1, \dots , x_1^{e_k}b_k\big] & \text{if $k$ even}, \\
0 & \text{if $k$ odd}.
\end{cases}
\end{equation}

\begin{proof}[Proof of \Cref{prop-bialg-DeFin}]
We check that equation~\eqref{eq-inv-bialg} is satisfied. For the odd moments we clearly have
\[
\sum_{1\le i_1,\dots,i_{2k+1}\le n} u^{e_1}_{j_1i_1}\cdots u^{e_{2k+1}}_{j_{2k+1}i_{2k+1}}\varphi\big(x^{e_1}_{i_1}\cdots x^{e_{2k+1}}_{i_{2k+1}}\big) = 0 = \varphi\big(x^{e_1}_{j_1}\cdots x^{e_{2k+1}}_{j_{2k+1}}\big)1.
\]
Let $j=(j_1\dots,j_{2k})\in\{1,\dots,n\}^{2k}$ and $e=(e_1,\dots,e_{2k})\in\{\varnothing,*\}^{2k}$.
Then
\begin{align*}
&\sum_{1\le i_1,\dots,i_{2k}\le n} u^{e_1}_{j_1i_1}\cdots u^{e_{2k}}_{j_{2k}i_{2k}}\varphi\big(x^{e_1}_{i_1}\cdots x^{e_{2k}}_{i_{2k}}\big) \\
&\quad{}= \sum_{1\le i_1,\dots,i_{2k}\le n} u^{e_1}_{j_1i_1}\cdots u^{e_{2k}}_{j_{2k}i_{2k}}\varphi\big(E\big(x^{e_1}_{i_1}\cdots x^{e_{2k}}_{i_{2k}}\big)\big) \\
&\quad{}= \sum_{1\le i_1,\dots,i_{2k}\le n} u^{e_1}_{j_1i_1}\cdots u^{e_{2k}}_{j_{2k}i_{2k}}\varphi\Bigg(\sum_{\pi\in \operatorname{NC}_2^{e,(\varnothing,*)},\,\pi\preceq \operatorname{ker} i} \kappa^E_{\pi}\big(x^{e_1}_{i_1}, \dots , x^{e_{2k}}_{i_{2k}}\big)\Bigg) \\
&\quad{}= \sum_{\pi\in \operatorname{NC}_2^{e,(\varnothing,*)}} \sum_{1\le i_1,\dots,i_{2k}\le n,\, \pi\preceq \operatorname{ker} i} u^{e_1}_{j_1i_1}\cdots u^{e_{2k}}_{j_{2k}i_{2k}}\varphi\big( \kappa^E_{\pi}\big(x^{e_1}_{i_1}, \dots , x^{e_{2k}}_{i_{2k}}\big)\big) \\
&\quad{}=\sum_{\pi\in \operatorname{NC}_2^{e,(\varnothing,*)}} \varphi\big( \kappa^E_{\pi}\big(x^{e_1}_{j_1}, \dots , x^{e_{2k}}_{j_{2k}}\big)\big) 1
= \varphi\big(x^{e_1}_{j_1}\cdots x^{e_k}_{j_k}\big)1,
\end{align*}
where we used first \Cref{lem-rel-Brown}, and then equation~\eqref{eq-cum}.
\end{proof}

\subsection*{Acknowledgements}
M.W.\ is supported by \emph{SFB-TRR 195} and \emph{DFG Heisenbergprogramm}.
I.B.\ and U.F.\ are supported by an ANR project (No. ANR-19-CE40-0002). G.C.\ is supported by the Project MESA (ANR-18-CE40-006) and by the Project STARS (ANR-20-CE40-0008) of the French National Research Agency (ANR).
We acknowledge the DAAD Procope program held by Roland Vergnioux and the fifth author from 2019--2020.

\pdfbookmark[1]{References}{ref}
\LastPageEnding

\end{document}